\documentclass[shortlabels]{CSML}
\pdfoutput=1

\usepackage{lastpage}
\newcommand{\dOi}{14(1:16)2018}
\lmcsheading{}{1--\pageref{LastPage}}{}{}%
{Aug.~01, 2017}{Feb.~27, 2018}{}

\usepackage{amssymb, amsmath}
\usepackage{microtype, stmaryrd, url,mathtools} 
\usepackage{mathtools}
\usepackage[hidelinks]{hyperref}

\DeclareMathAlphabet{\mathbf}{OT1}{cmr}{b}{n}

\usepackage[arrow, matrix, tips, curve, graph, rotate]{xy}
\SelectTips{cm}{10}

\makeatletter
\def\matrixobject@{%
  \edef \next@{={\DirectionfromtheDirection@ }}%
  \expandafter \toks@ \next@ \plainxy@
  \let\xy@@ix@=\xyq@@toksix@  
  \xyFN@ \OBJECT@}
\let\xy@entry@@norm=\entry@@norm
\def\entry@@norm@patched{%
  \let\object@=\matrixobject@
  \xy@entry@@norm }
\AtBeginDocument{\let\entry@@norm\entry@@norm@patched}
\makeatother

\newcommand{\twocong}[2][0.5]{\ar@{}[#2] \save ?(#1)*{\cong}\restore}
\newcommand{\twosim}[2][0.5]{\ar@{}[#2] \save ?(#1)*{\simeq}\restore}
\newcommand{\twoeq}[2][0.5]{\ar@{}[#2] \save ?(#1)*{=}\restore}
\newcommand{\rtwocell}[3][0.5]{\ar@{}[#2] \ar@{=>}?(#1)+/l 0.2cm/;?(#1)+/r 0.2cm/^{#3}}
\newcommand{\rtwocello}[3][0.5]{\ar@{}[#2] \ar@{=>}?(#1)+/l 0.2cm/;?(#1)+/r 0.2cm/_{#3}}
\newcommand{\ltwocell}[3][0.5]{\ar@{}[#2] \ar@{=>}?(#1)+/r 0.2cm/;?(#1)+/l 0.2cm/^{#3}}
\newcommand{\ltwocello}[3][0.5]{\ar@{}[#2] \ar@{=>}?(#1)+/r 0.2cm/;?(#1)+/l 0.2cm/_{#3}}
\newcommand{\dtwocell}[3][0.5]{\ar@{}[#2] \ar@{=>}?(#1)+/u  0.2cm/;?(#1)+/d 0.2cm/^{#3}}
\newcommand{\dltwocell}[3][0.5]{\ar@{}[#2] \ar@{=>}?(#1)+/ur  0.2cm/;?(#1)+/dl 0.2cm/^{#3}}
\newcommand{\drtwocell}[3][0.5]{\ar@{}[#2] \ar@{=>}?(#1)+/ul  0.2cm/;?(#1)+/dr 0.2cm/^{#3}}
\newcommand{\dthreecell}[3][0.5]{\ar@{}[#2] \ar@3{->}?(#1)+/u  0.2cm/;?(#1)+/d 0.2cm/^{#3}}
\newcommand{\utwocell}[3][0.5]{\ar@{}[#2] \ar@{=>}?(#1)+/d 0.2cm/;?(#1)+/u 0.2cm/_{#3}}
\newcommand{\dtwocelltarg}[3][0.5]{\ar@{}#2 \ar@{=>}?(#1)+/u  0.2cm/;?(#1)+/d 0.2cm/^{#3}}
\newcommand{\utwocelltarg}[3][0.5]{\ar@{}#2 \ar@{=>}?(#1)+/d  0.2cm/;?(#1)+/u 0.2cm/_{#3}}

\newdir{(}{{}*!<0em,-.14em>-\cir<.14em>{l^r}}
\newdir{ (}{{}*!/-5pt/\dir{(}}
\newdir{ >}{{}*!/-5pt/\dir{>}}
\newcommand{\sh}[2]{**{!/#1 -#2/}}

\DeclareMathOperator{\ob}{ob}

\newcommand{\cat}[1]{\mathbf{#1}}
\newcommand{\op}{\mathrm{op}}

\newcommand{\id}{\mathrm{id}}
\newcommand{\thg}{{\mathord{\text{--}}}}

\newcommand{\comp}{\otimes}
\newcommand{\unit}{I}
\newcommand{\lp}{\cat{LexProf}}
\newcommand{\lfp}{\cat{LFP}}
\newcommand{\fl}{\cat{FL}}
\newcommand{\pfl}{\cat{PARFL}}

\newcommand{\af}{\smash{{\A_{f\,}}\!^\mathrm{op}}}
\newcommand{\ic}{\textstyle\int\!\C}

\newcommand{\spn}[1]{{\langle{#1}\rangle}}

\newcommand{\defeq}{\mathrel{\mathop:}=}
\newcommand{\cd}[2][]{\vcenter{\hbox{\xymatrix#1{#2}}}}

\renewcommand{\phi}{\varphi}
\newcommand{\A}{{\mathcal A}}

\newcommand{\C}{{\mathcal C}}
\newcommand{\D}{{\mathcal D}}

\newcommand{\F}{{\mathcal F}}

\newcommand{\I}{{\mathcal I}}

\renewcommand{\L}{{\mathcal L}}

\renewcommand{\S}{{\mathcal S}}
\newcommand{\T}{{\mathcal T}}

\newcommand{\V}{{\mathcal V}}
\newcommand{\W}{{\mathcal W}}

\newcommand{\xtor}[1]{\cdl[@1]{{} \ar[r]|-{\object@{|}}^{#1} & {}}}
\newcommand{\tor}{\ensuremath{\relbar\joinrel\mapstochar\joinrel\rightarrow}}

\makeatletter

\def\hookleftarrowfill@{\arrowfill@\leftarrow\relbar{\relbar\joinrel\rhook}}
\def\twoheadleftarrowfill@{\arrowfill@\twoheadleftarrow\relbar\relbar}
\def\leftbararrowfill@{\arrowdoublefill@{\leftarrow\mkern-5mu}\relbar\mapstochar\relbar\relbar}
\def\Leftbararrowfill@{\arrowdoublefill@{\Leftarrow\mkern-2mu}\Relbar\Mapstochar\Relbar\Relbar}
\def\leftringarrowfill@{\arrowdoublefill@{\leftarrow\mkern-3mu}\relbar{\mkern-3mu\circ\mkern-2mu}\relbar\relbar}
\def\lefttriarrowfill@{\arrowfill@{\mathrel\triangleleft\mkern0.5mu\joinrel\relbar}\relbar\relbar}
\def\Lefttriarrowfill@{\arrowfill@{\mathrel\triangleleft\mkern1mu\joinrel\Relbar}\Relbar\Relbar}

\def\hookrightarrowfill@{\arrowfill@{\lhook\joinrel\relbar}\relbar\rightarrow}
\def\twoheadrightarrowfill@{\arrowfill@\relbar\relbar\twoheadrightarrow}
\def\rightbararrowfill@{\arrowdoublefill@{\relbar\mkern-0.5mu}\relbar\mapstochar\relbar\rightarrow}
\def\Rightbararrowfill@{\arrowdoublefill@{\Relbar\mkern-2mu}\Relbar\Mapstochar\Relbar\Rightarrow}
\def\rightringarrowfill@{\arrowdoublefill@\relbar\relbar{\mkern-2mu\circ\mkern-3mu}\relbar{\mkern-3mu\rightarrow}}
\def\righttriarrowfill@{\arrowfill@\relbar\relbar{\relbar\joinrel\mkern0.5mu\mathrel\triangleright}}
\def\Righttriarrowfill@{\arrowfill@\Relbar\Relbar{\Relbar\joinrel\mkern1mu\mathrel\triangleright}}

\def\leftrightarrowfill@{\arrowfill@\leftarrow\relbar\rightarrow}
\def\mapstofill@{\arrowfill@{\mapstochar\relbar}\relbar\rightarrow}

\renewcommand*\xleftarrow[2][]{\ext@arrow 20{20}0\leftarrowfill@{#1}{#2}}
\providecommand*\xLeftarrow[2][]{\ext@arrow 60{22}0{\Leftarrowfill@}{#1}{#2}}
\providecommand*\xhookleftarrow[2][]{\ext@arrow 10{20}0\hookleftarrowfill@{#1}{#2}}
\providecommand*\xtwoheadleftarrow[2][]{\ext@arrow 60{20}0\twoheadleftarrowfill@{#1}{#2}}
\providecommand*\xleftbararrow[2][]{\ext@arrow 10{22}0\leftbararrowfill@{#1}{#2}}
\providecommand*\xLeftbararrow[2][]{\ext@arrow 50{24}0\Leftbararrowfill@{#1}{#2}}
\providecommand*\xleftringarrow[2][]{\ext@arrow 10{26}0\leftringarrowfill@{#1}{#2}}
\providecommand*\xlefttriarrow[2][]{\ext@arrow 80{24}0\lefttriarrowfill@{#1}{#2}}
\providecommand*\xLefttriarrow[2][]{\ext@arrow 80{24}0\Lefttriarrowfill@{#1}{#2}}

\renewcommand*\xrightarrow[2][]{\ext@arrow 01{20}0\rightarrowfill@{#1}{#2}}
\providecommand*\xRightarrow[2][]{\ext@arrow 04{22}0{\Rightarrowfill@}{#1}{#2}}
\providecommand*\xhookrightarrow[2][]{\ext@arrow 00{20}0\hookrightarrowfill@{#1}{#2}}
\providecommand*\xtwoheadrightarrow[2][]{\ext@arrow 03{20}0\twoheadrightarrowfill@{#1}{#2}}
\providecommand*\xrightbararrow[2][]{\ext@arrow 01{22}0\rightbararrowfill@{#1}{#2}}
\providecommand*\xRightbararrow[2][]{\ext@arrow 04{24}0\Rightbararrowfill@{#1}{#2}}
\providecommand*\xrightringarrow[2][]{\ext@arrow 01{26}0\rightringarrowfill@{#1}{#2}}
\providecommand*\xrighttriarrow[2][]{\ext@arrow 07{24}0\righttriarrowfill@{#1}{#2}}
\providecommand*\xRighttriarrow[2][]{\ext@arrow 07{24}0\Righttriarrowfill@{#1}{#2}}

\providecommand*\xmapsto[2][]{\ext@arrow 01{20}0\mapstofill@{#1}{#2}}
\providecommand*\xleftrightarrow[2][]{\ext@arrow 10{22}0\leftrightarrowfill@{#1}{#2}}
\providecommand*\xLeftrightarrow[2][]{\ext@arrow 10{27}0{\Leftrightarrowfill@}{#1}{#2}}

\makeatother

\begin{document}
\leftmargini=2em \title[An enriched view on the monad--theory
correspondence]{An enriched view on the extended\\finitary
  monad--Lawvere theory correspondence}

\author[R.~Garner]{Richard Garner}
\address{Department of Mathematics, Macquarie University, NSW 2109,
  Australia}
\email{richard.garner@mq.edu.au}

\author[J.~Power]{John Power}
\address{Department of Computer Science, University of Bath, Claverton
  Down, Bath BA2 7AY, United Kingdom}
\email{A.J.Power@bath.ac.uk}




\thanks{The authors gratefully acknowledge the support of Australian
  Research Council grants DP160101519 and FT160100393 and of the Royal
  Society International Exchanges scheme, grant number IE151369. No
  data was generated in the research for this paper.}

\begin{abstract}
  We give a new account of the correspondence, first established by
  Nishizawa--Power, between finitary monads and Lawvere theories over
  an arbitrary locally finitely presentable base. Our account explains
  this correspondence in terms of enriched category theory: the
  passage from a finitary monad to the corresponding Lawvere theory is
  exhibited as an instance of \emph{free completion} of an enriched
  category under a class of absolute colimits. This extends work of
  the first author, who established the result in the special case of
  finitary monads and Lawvere theories over the category of sets; a
  novel aspect of the generalisation is its use of enrichment over a
  \emph{bicategory}, rather than a monoidal category, in order to
  capture the monad--theory correspondence over all locally finitely
  presentable bases simultaneously.
\end{abstract}

\dedicatory{Dedicated to Ji\v r\' i Ad\' amek on the occasion of his retirement}
\maketitle

\section{Introduction}
\label{sec:introduction}

A key theme of Ji\v r\' i Ad\' amek's superlative research career has
been the study of the subtle interaction between monads and theories,
especially within computer science. We hope he might see this paper as
a development of the abstract mathematics underlying this aspect of
his body of work. Ji\v r\' i has been an inspiration to both of us, on
both a scientific and a personal level, as he has been to many in
category theory and beyond, and we are therefore pleased to
dedicate this paper to him.

The starting point of our development is the well-known fact that
categorical universal algebra provides two distinct ways to approach
the notion of (single-sorted, finitary) equational algebraic theory.
On the one hand, any such theory $\mathbb{T}$ gives rise to a Lawvere
theory whose models coincide (to within coherent isomorphism) with the
$\mathbb{T}$-models. Recall that a \emph{Lawvere theory} is a small
category $\L$ equipped with an identity-on-objects, strict
finite-power-preserving functor
$\mathbb{F}^\mathrm{op} \rightarrow \L$ and that a \emph{model} of a
Lawvere theory is a finite-power-preserving functor
$\L \rightarrow \cat{Set}$, where here $\mathbb{F}$ denotes the
category of finite cardinals and mappings.

On the other hand, an algebraic theory $\mathbb{T}$ gives rise to a
finitary (i.e., filtered-colimit-preserving) monad on the category of
sets, whose Eilenberg--Moore algebras also coincide with the
$\mathbb T$-models. We can pass back and forth between the
presentations using finitary monads and Lawvere theories in a manner
compatible with semantics; this is encapsulated by an equivalence of
categories fitting into a triangle
\begin{equation}\label{eq:8}
  \cd[@!C@C-3.5em]{
    \cat{Mnd}_f(\cat{Set})^\mathrm{op}
    \ar[dr]_-{\cat{Alg}(\thg)}\ar[rr]^(0.5){\simeq} & \twosim{d} & \cat{Law}^\mathrm{op}
    \ar[dl]^-{\cat{Mod}(\thg)} \\ & \cat{CAT} &
  }
\end{equation}
which commutes up to pseudonatural equivalence. Both these categorical
formulations of equational algebraic structure are \emph{invariant}
with respect to the models, meaning that whenever two algebraic
theories have isomorphic categories of models, the Lawvere theories
and the monads they induce are also isomorphic. However, the two
approaches emphasise different aspects of an equational theory
$\mathbb{T}$. On the one hand, the Lawvere theory $\L$ encapsulates
the \emph{operations} of $\mathbb{T}$: the hom-set $\L(m,n)$ comprises
all the operations $X^m \rightarrow X^n$ which are definable in any
$\mathbb{T}$-model $X$. On the other hand, the action of the monad $T$
encapsulates the construction of the \emph{free $\mathbb{T}$-model} on
any set; though since there are infinitary equational theories which
also admit free models, the restriction to finitary monads is
necessary to recover the equivalence with Lawvere theories.

While the equivalence in~\eqref{eq:8} is not hard to construct, there
remains the question of how it should be understood. Clarifying this
point was the main objective of~\cite{Garner2014Lawvere}: it describes
a setting within which both finitary monads on $\cat{Set}$ and Lawvere
theories can be considered on an equal footing, and in which the
passage from a finitary monad to the associated Lawvere theory can be
understood as an instance of the same process by which one associates:
\begin{enumerate}[(a)]
\item To a locally small category $\C$, its Karoubian envelope;
\item To a ring $R$, its category of
  finite-dimensional matrices;
\item To a metric space, its Cauchy-completion.
\end{enumerate}

The setting is that of $\V$-enriched category
theory~\cite{Kelly1982Basic}; while the process is that of free
completion under a class of \emph{absolute
  colimits}~\cite{Street1983Absolute}---colimits that are preserved by
any $\V$-functor. The above examples are instances of such a
completion, since:
\begin{enumerate}[(a),itemsep=0.25\baselineskip]
\item Each locally small category is a $\cat{Set}$-category, and splittings
  of idempotents are $\cat{Set}$-absolute colimits;
\item Each ring can be seen as a one-object $\cat{Ab}$-category and
  the corresponding
  category of finite-dimensional matrices can be obtained by freely adjoining finite
  biproducts---which are $\cat{Ab}$-absolute colimits;
\item Each metric space can be seen as an
  $\mathbb{R}_{+}^\infty$-category---where $\mathbb{R}_+^\infty$ is
  the monoidal poset of non-negative reals extended by infinity, as
  defined in~\cite{Lawvere1973Metric}---and its Cauchy-completion can
  be obtained by adding limits for Cauchy sequences which, again as
  in~\cite{Lawvere1973Metric}, are $\mathbb{R}_{+}^\infty$-absolute
  colimits.
\end{enumerate}

In order to fit~\eqref{eq:8} into this same setting, one takes as
enrichment base the category $\F$ of finitary endofunctors of
$\cat{Set}$, endowed with its composition monoidal structure. On the
one hand, finitary monads on $\cat{Set}$ are the same as monoids in
$\F$, which are the same as one-object $\F$-categories. On the other
hand, Lawvere theories may also be identified with certain
$\F$-categories; the argument here is slightly more involved, and may
be summarised as follows.

A key result of~\cite{Garner2014Lawvere}, recalled as
Proposition~\ref{prop:11} below, identifies $\F$-categories admitting
all absolute \emph{tensors} (a kind of enriched colimit) with ordinary
categories admitting all finite powers. In one direction, we obtain a
category with finite powers from an absolute-tensored $\F$-category by
taking the underlying ordinary category; in the other, we use a
construction which generalises the \emph{endomorphism monad} (or in
logical terms the \emph{complete theory}) of an object in a category
with finite powers.

Using this key result, we may identify Lawvere theories with
identity-on-objects strict absolute-tensor-preserving $\F$-functors
$\mathbb{F}^\mathrm{op} \rightarrow \L$, where, overloading notation,
we use $\mathbb{F}^\mathrm{op}$ and $\L$ to denote not just the
relevant categories with finite powers, but also the corresponding
$\F$-categories. In~\cite{Garner2014Lawvere}, the $\F$-categories
equipped with an $\F$-functor of the above form were termed
\emph{Lawvere $\F$-categories}.

By way of the above identifications, the equivalence~\eqref{eq:8}
between finitary monads and Lawvere theories can now be re-expressed
as an equivalence between one-object $\F$-categories and Lawvere
$\F$-categories: which can be obtained via the standard
enriched-categorical process of \emph{free completion} under all
absolute tensors. The universal property of this completion also
explains the compatibility with the semantics in~\eqref{eq:8}; we
recall the details of this in Section~\ref{sec:one-object-case} below.

We thus have three categorical perspectives on equational algebraic
theories: as Lawvere theories, as finitary monads on $\cat{Set}$, and
(encompassing the other two) as $\F$-categories. It is natural to ask
if these perspectives extend so as to account for algebraic structure
borne not by \emph{sets} but by objects of an arbitrary locally
finitely presentable category $\A$. This is of practical interest,
since such structure arises throughout mathematics and computer
science, as in, for example, sheaves of rings, or monoidal categories,
or the second-order algebraic structure of~\cite{Fiore1999Abstract}.

The approach using monads extends easily: we simply replace finitary
monads on $\cat{Set}$ by finitary monads on $\A$. The approach using
Lawvere theories also extends, albeit more delicately, by way of the
\emph{Lawvere $\A$-theories} of~\cite{Nishizawa2009Lawvere,
  Lack2009Gabriel-Ulmer}. If we write $\A_f$ for a skeleton of the
full subcategory of finitely presentable objects in $\A$, then a
Lawvere $\A$-theory is a small category $\L$ together with an
identity-on-objects finite-limit-preserving\footnote{Note the absence
  of the qualifier ``strict''; for a discussion of this, see
  Remark~\ref{rk:2} below.} functor $J \colon \af \rightarrow \L$;
while a model of this theory is a functor $\L \rightarrow \cat{Set}$
whose restriction along $J$ preserves finite limits. These definitions
are precisely what is needed to extend the equivalence~\eqref{eq:8} to
one of the form:
\begin{equation}\label{eq:7}
  \cd[@!C@C-3em]{
    \cat{Mnd}_f(\A)^\mathrm{op}
    \ar[dr]_-{\cat{Alg}(\thg)}\ar[rr]^(0.5){\simeq} & \twosim{d} & \cat{Law}(\A)^\mathrm{op}\rlap{ .}
    \ar[dl]^-{\cat{Mod}(\thg)} \\ & \cat{CAT} &
  }
\end{equation}

What does not yet exist in this situation is an extension of the
third, enriched-categorical perspective; the objective of this paper
is to provide one. Like in~\cite{Garner2014Lawvere}, this will provide
a common setting in which the approaches using monads and using
Lawvere theories can coexist; and, like before, it will provide an
explanation as to \emph{why} an equivalence~\eqref{eq:7} should exist,
by exhibiting it as another example of a completion under a class of
absolute colimits.

There is a subtlety worth remarking on in how we go about this. One
might expect that, for each locally finitely presentable $\A$, one
simply replaces the monoidal category $\F$ of finitary endofunctors of
$\cat{Set}$ by the monoidal category $\F_\A$ of finitary endofunctors
of $\A$, and then proceeds as before. This turns out not to work in
general: there is a paucity of $\F_\A$-enriched absolute colimits,
such that freely adjoining them to a one-object $\F_\A$-category does
not necessary yield something resembling the notion of Lawvere
$\A$-theory.

In overcoming this apparent obstacle, we are led to a global analysis
which is arguably more elegant: it involves a single
enriched-categorical setting in which finitary monads and Lawvere
theories over all locally finitely presentable bases coexist
simultaneously, and in which the monad--theory correspondences for
each $\A$ arise as instances of the same free completion
process.

This setting involves enrichment not in the monoidal category of
finitary endofunctors of a particular $\A$, but in the
\emph{bicategory} $\lfp$ of finitary functors between locally
presentable categories\footnote{In fact, when it comes down to it, we
  will work not with $\lfp$ itself, but with a biequivalent bicategory
  $\lp$ of \emph{lex profunctors}; see
  Section~\ref{sec:lex-prof-finit}.}. The theory of categories
enriched in a bicategory was developed
in~\cite{Walters1981Sheaves,Street1983Enriched} and will be recalled
in Section~\ref{sec:ingr-gener} below; for now, note that a one-object
$\lfp$-enriched category is a monad in $\lfp$, thus, a finitary monad
on a locally finitely presentable category. This explains one side of
the correspondence~\eqref{eq:7}; for the other, we extend the key
technical result of~\cite{Garner2014Lawvere} by showing that
absolute-tensored $\lfp$-categories can be identified with what we
call \emph{partially finitely complete} categories. These will be
defined in Section~\ref{sec:part-finite-completeness} below; they are
categories $\C$, not necessarily finitely complete, that come equipped
with a sieve of finite-limit-preserving functors expressing which
finite limits do in fact exist in $\C$.

The relevance of this result is as follows. If
$J \colon \af \rightarrow \L$ is a Lawvere $\A$-theory, then we may
view both $\af$ and $\L$ as partially finitely complete, and so as
absolute-tensored $\lfp$-categories, on equipping the former with the
sieve of \emph{all} finite-limit-preserving functors into it, and the
latter with the sieve of all finite-limit-preserving functors which
factor through $J$. In this way, we can view a Lawvere $\A$-theory as
a particular kind of $\lfp$-functor $\af \rightarrow \L$, where, as
before, we overload notation by using the same names for the
$\lfp$-enriched categories as for the ordinary categories from which
they are derived.

If we term the $\lfp$-functors arising in this way \emph{Lawvere
  $\lfp$-categories}, then our reconstruction of the
equivalence~\eqref{eq:7} will follow, exactly as
in~\cite{Garner2014Lawvere}, upon showing the equivalence of
one-object $\lfp$-categories, and Lawvere $\lfp$-categories; and,
exactly as before, we will obtain this equivalence via the
enriched-categorical process of \emph{free completion} under absolute
tensors. Moreover, the universal property of this free completion once
again explains the compatibility of this equivalence with the taking
of semantics; thereby giving an enriched-categorical explanation of
the entire triangle~\eqref{eq:7}.

\section{The one-object case}
\label{sec:one-object-case}

In this section, we summarise and discuss the manner in
which~\cite{Garner2014Lawvere} reconstructs the equivalence of
finitary monads on $\cat{Set}$ and Lawvere theories from an
enriched-categorical viewpoint. Most of what we will say is simply
revision, but note that the points clarified by
Proposition~\ref{prop:2} and Example~\ref{ex:4} are new.

We start from the observation that finitary monads on $\cat{Set}$ are
equally monoids in the monoidal category $\F$ of finitary endofunctors
of $\cat{Set}$, so equally one-object \emph{$\F$-enriched categories}
in the sense of~\cite{Kelly1982Basic}. More precisely:
\begin{prop}
  \label{prop:13}
  The category of finitary monads on $\cat{Set}$ is equivalent to the
  category of one-object $\F$-enriched categories.
\end{prop}

To understand Lawvere theories from the $\F$-enriched perspective is a
little more involved. As a first step, note that if
$J \colon \mathbb{F}^\mathrm{op} \rightarrow \L$ is a Lawvere theory,
then both $\mathbb{F}^\mathrm{op}$ and $\L$ are categories with finite
powers, and $J$ is a finite-power-preserving functor between them.
Thus the desired understanding flows from one of the key results
of~\cite{Garner2014Lawvere}, which shows that categories admittings
finite powers are equivalent to $\F$-enriched categories admitting all
\emph{absolute tensors} in the following sense.

\begin{defi}
  If $\V$ is a monoidal category and $\C$ is a $\V$-category, then a
  \emph{tensor} of $X \in \C$ by $V \in \V$ comprises an object
  $V \cdot X \in \C$ and morphism
  $u \colon V \rightarrow \C(X, V \cdot X)$ in $\V$, such that, for
  any $U \in \V$ and $Y \in \C$, the assignation
  \begin{equation}\label{eq:16}
    U \xrightarrow{f} \C(V \cdot X, Y) \ \ \mapsto \ \ 
    U \otimes V \xrightarrow{f \otimes u} \C(V \cdot X, Y) \otimes \C(X,
    V \cdot X) \xrightarrow{\circ} \C(X, Y)
  \end{equation}
  gives a bijection between maps $U \rightarrow \C(V \cdot X, Y)$ and
  $U \otimes V \rightarrow \C(X,Y)$ in $\V$. Such a tensor is said to
  be \emph{preserved} by a $\V$-functor $F \colon \C \rightarrow \D$ if
  the composite morphism
  $F_{X,V\cdot X} \circ u \colon V \rightarrow \C(X,V\cdot X)
  \rightarrow \D(FX, F(V \cdot X))$ in $\V$ exhibits $F(V \cdot X)$ as
  $V \cdot FX$. Tensors by $V \in \V$ are said to be \emph{absolute} if
  they are preserved by any $\V$-functor.
\end{defi}

There is a delicate point here. The theory of \cite{Kelly1982Basic}
considers enrichment only over a symmetric monoidal closed base; by
contrast, our $\F$ is non-symmetric and \emph{right-closed}---meaning
that there exists a right adjoint $[V, \thg]$ to the functor
$(\thg) \otimes V$ tensoring on the \emph{right} by an object $V$. The
value $[V,W]$ of this right adjoint can be computed by first forming
the right Kan extension $\mathrm{Ran}_V W \in [\cat{Set}, \cat{Set}]$,
and then the finitary coreflection of that. On the other hand, $\F$ is
not \emph{left-closed}---meaning that there is not always a right
adjoint to the functor $V \otimes (\thg)$ tensoring on the \emph{left}
by $V$---because $V \otimes (\thg)$ will not be cocontinuous if $V$
itself is not cocontinuous.

While the non-symmetric, right-closed setting is too weak to allow
constructions such as opposite $\F$-categories, functor
$\F$-categories, or tensor product of $\F$-categories, it is strong
enough to allow for a good theory of $\F$-enriched colimits---of which
absolute tensors are an example. In particular, we may give the
following tractable characterisation of the absolute tensors over a
right-closed base. In the statement, recall that a \emph{left dual}
for $V \in \V$ is a left adjoint for $V$ seen as a $1$-cell in the
one-object bicategory corresponding to $\V$.

\begin{prop}\label{prop:2}
  Let $\V$ be a right-closed monoidal category. Tensors by $V \in \V$
  are absolute if and only if $V$ admits a left dual in $\V$.
\end{prop}

This result originates in~\cite{Street1983Absolute}, though some
adaptations to the proof are required in the non-left-closed setting;
we defer giving these to Section~\ref{sec:absolute-tensors-w} below,
where we will give a proof which works in the more general
bicategory-enriched context.

For now, applying this result when $\V = \F$, we see that tensors by
an object $F \in \F$---a finitary endofunctor of $\cat{Set}$---are
absolute just when $F$ has a (necessarily finitary) left adjoint $G$.
In this case, by adjointness we must have $F \cong (\thg)^{G1}$;
moreover, in order for $F$ to be a finitary endofunctor, $G1$ must be
finitely presentable in $\cat{Set}$, thus, a finite set. So an
$\F$-category is absolute-tensored just when it admits tensors by
$(\thg)^n \in \F$ for all $n \in \mathbb{N}$.
In~\cite{Garner2014Lawvere}, such $\F$-categories were called   
\emph{representable}.

The following further characterisation of the absolute-tensored
$\F$-categories is Proposition~3.8 of~\emph{ibid}.; in the statement,
we write $\F\text-\cat{CAT}_\mathrm{abs}$ for the $2$-category of
absolute-tensored $\F$-categories and all (necessarily
absolute-tensor-preserving) $\F$-functors and $\F$-transformations,
and write $\cat{FPOW}$ for the $2$-category of categories with finite
powers and finite-power-preserving functors.

\begin{prop}
  \label{prop:11}
  The underlying ordinary category of any absolute-tensored
  \mbox{$\F$-category} admits finite powers, while the underlying ordinary
  functor of any \mbox{$\F$-functor} between representable $\F$-categories
  preserves finite powers. The induced $2$-functor
  $\F\text-\cat{CAT}_\mathrm{abs} \rightarrow \cat{FPOW}$ is an
  equivalence of $2$-categories.
\end{prop}

This result is the technical heart of~\cite{Garner2014Lawvere}; as
there, the task of giving its proof will be eased if we replace the
category of finitary endofunctors of $\cat{Set}$ by the equivalent
functor category $[\mathbb{F}, \cat{Set}]$. The equivalence in
question arises via left Kan extension and restriction along the
inclusion $\mathbb{F} \rightarrow \cat{Set}$; and transporting the
composition monoidal structure on finitary endofunctors across it
yields the so-called \emph{substitution} monoidal structure on
$[\mathbb{F}, \cat{Set}]$, with tensor and unit given by
$(A \otimes B)(n) = \smash{\int^k Ak \times (Bn)^k}$ and $I(n) = n$.
Henceforth, we re-define $\F$ to be this monoidal category. Having
done so, we see that a general $\F$-category $\C$ involves objects
$X,Y,\dots$, hom-objects $\C(X,Y) \in [\mathbb{F}, \cat{Set}]$, and
composition and identities notated as follows:
\begin{equation}
  \label{eq:37}
  \begin{aligned}
    \textstyle\int^k \C(Y,Z)(k) \times \C(X,Y)(n)^k & \rightarrow
    \C(X,Z)(n) & I(n) & \rightarrow
    \C(X,X)(n) \\
    [g, f_1, \dots, f_k] & \mapsto g \circ (f_1, \dots, f_k) & i &
    \mapsto \pi_i\rlap{ .}
  \end{aligned}
\end{equation}
Note moreover that, since the unit $I \in [\mathbb{F}, \cat{Set}]$ is
represented by $1$, the arrows $X \rightarrow Y$ in the underlying
ordinary category $\C_0$ of $\C$ are the elements of $\C(X,Y)(1)$.

\begin{proof}[Proof of Proposition~\ref{prop:11}]
  Suppose $\C$ is an absolute-tensored $\F$-category. Then for each
  $X \in \C$ and $n \in \mathbb{N}$, we have $y_n \cdot X \in \C$ and
  a unit map $y_n \rightarrow \C(X, y_n \cdot X)$, or equally, an
  element $u \in \C(X, y_n \cdot X)(n)$, rendering each~\eqref{eq:16}
  invertible. When $U = y_1$, the function~\eqref{eq:16} is given, to
  within isomorphism, by
  \begin{equation}
    \label{eq:39}
    \begin{aligned}
      \C(y_n \cdot X,Y)(1) & \rightarrow \C(X,Y)(n) \\
      f & \mapsto f \circ (u)\rlap{ ;}
    \end{aligned}
  \end{equation}
  thus, when $Y = X$ we obtain elements
  $p_1, \dots, p_n \in \C(y_n \cdot X, X)(1)$ with
  $p_i \circ (u) = \pi_i$ in $\C(X,X)(n)$. It follows by the
  $\F$-category axioms that
  $(u \circ (p_1, \dots, p_n)) \circ (u) = u$, and so, by~\eqref{eq:39}
  with $Y = y_n \cdot X$, that
  $u \circ (p_1, \dots, p_n) = \id_{y_n \cdot X}$ in
  $\C(y_n \cdot X, y_n \cdot X)(1)$.

  We claim that the maps $p_i \colon y_n \cdot X \to X$ in $\C_0$
  exhibit $y_n \cdot X$ as the $n$-fold power $X^n$. Indeed, given
  $g_1, \dots, g_n \colon Z \rightarrow X$ in $\C_0$, we define
  $g \colon Z \rightarrow y_n \cdot X$ by
  $g = u \circ (g_1, \dots, g_n)$, and now $p_i \circ (u) = \pi_i$
  implies $p_i \circ (g) = g_i$. Moreover, if
  $h \colon Z \rightarrow y_n \cdot X$ satisfies $p_i \circ (h) = g_i$
  then
  $g = u \circ (p_1 \circ (h), \dots, p_n \circ (h)) = (u \circ (p_1,
  \dots, p_n)) \circ h = h$.

  This proves the first claim. The second follows easily from the fact
  that any $\F$-functor preserves absolute tensors, and so we have a
  $2$-functor $\F\text-\cat{CAT}_\mathrm{abs} \rightarrow \cat{FPOW}$.
  Finally, to show this is a $2$-equivalence we construct an explicit
  pseudoinverse. To each category $\D$ with finite powers, we associate
  the $\F$-category $\underline \D$ with objects those of $\D$,
  with hom-objects $\underline \D(X, Y) = \D(X^{(\thg)}, Y)$, with
  composition operations~\eqref{eq:37} obtained using the universal
  property of power, and with identity elements $\pi_i$ given by power projections.
  This $\underline \D$ is absolute-tensored on taking the tensor of
  $X \in \underline \D$ by $y_n$ to be $X^n$, with unit element
  $1_{X^n} \in \underline \D(X, X^n)(n) = \D(X^n, X^n)$. It is now
  straightforward to extend the assignation $\D \mapsto \underline \D$
  to a $2$-functor
  $\cat{FPOW} \rightarrow \F\text-\cat{CAT}_\mathrm{abs}$, and to show
  that this is pseudoinverse to the underlying ordinary category
  $2$-functor.
\end{proof}

In particular, if $J \colon \mathbb{F}^\mathrm{op} \rightarrow \L$ is
a Lawvere theory, then we can view both $\mathbb{F}^\mathrm{op}$ and
$\L$ as $\F$-categories, and $J$ as an $\F$-functor between them.
Defining a \emph{Lawvere $\F$-category} to be a representable
$\F$-category $\L$ equipped with an identity-on-objects, strict
absolute-tensor-preserving $\F$-functor
$\mathbb{F}^\mathrm{op} \rightarrow \L$, it follows easily that:

\begin{prop}
  \label{prop:14}
  The category of Lawvere $\F$-categories is equivalent to the
  category of Lawvere theories (where maps in each case are commuting
  triangles under $\mathbb{F}^\mathrm{op}$).
\end{prop}

Using Propositions~\ref{prop:13} and~\ref{prop:14}, we may now
re-express the the monad--theory correspondence~\eqref{eq:8} in
$\F$-categorical terms as an equivalence between one-object
$\F$-categories and Lawvere $\F$-categories. We obtain this using the
process of \emph{free completion} under absolute tensors---a
description of which can be derived
from~\cite{Betti1982Cauchy-completion}.

\begin{prop}
  \label{prop:12}
  Let $\V$ be a monoidal category and let $\V_d \subset \V$ be a
  subcategory equivalent to the full subcategory of objects with left
  duals in $\V$. The free completion under absolute tensors of a
  $\V$-category $\C$ is given by the $\V$-category $\bar \C$ with:
  \begin{itemize}
  \item Objects $V \cdot X$, where $X \in \C$ and $V \in \V_d$;
  \item Hom-objects
    $\bar \C(V \cdot X, W \cdot Y) = W \otimes \C(X,Y) \otimes V^\ast$
    (for $V^\ast$ a left dual for $V$);
  \item Composition built from composition in $\C$ and the counit maps
    $\varepsilon \colon W^\ast \otimes W \rightarrow I$;
  \item Identities built from the unit maps
    $\eta \colon I \rightarrow V \otimes V^\ast$ and identities in
    $\C$.
  \end{itemize}
\end{prop}

Taking $\V = \F$ and $\V_d$ to be the full subcategory on the $y_n$'s,
we thus arrive at:

\begin{prop}
  \label{prop:15}
  The category of one-object $\F$-categories is equivalent to the
  category of Lawvere $\F$-categories.
\end{prop}

\begin{proof}
  The free completion under absolute tensors of the unit $\F$-category
  $\I$ is $\mathbb{F}^\mathrm{op}$; whence each one-object
  $\F$-category $\C$ yields a Lawvere $\F$-category on applying
  completion under absolute tensors to the unique $\F$-functor
  $\I \rightarrow \C$. In the other direction, we send a Lawvere
  $\F$-category $J \colon \mathbb{F}^\mathrm{op} \rightarrow \L$ to
  the one-object sub-$\F$-category of $\L$ on $J1$.
\end{proof}

The conjunction of Propositions~\ref{prop:13},~\ref{prop:14}
and~\ref{prop:15} now yields the equivalence on the top row
of~\eqref{eq:8}. More pedantically, it yields \emph{an} equivalence,
which we should check is in fact the usual one:

\begin{exa}
  \label{ex:4}
  Let $T$ be a finitary monad on $\cat{Set}$, and let $\T$ be the
  corresponding one-object $\F$-category; thus, $\T$ has a single
  object $X$ with $\T(X, X)(n) = Tn$, and composition and identities
  coming from the monad structure of $T$. The free completion
  $\bar \T$ of $\T$ under absolute tensors has objects $y_n \cdot X$
  for $n \in \mathbb{N}$---or equally, just natural numbers---and
  hom-objects given by
  \begin{equation}\label{eq:33}
    \bar \T(n, m) = y_m \otimes \T(X,X) \otimes
    (y_n)^\ast \cong (T(n \times \thg))^m \rlap{ .}
  \end{equation}
  The underlying ordinary category of $\bar \T$ is thus the category
  $\L$ with natural numbers as objects, and $\L(n,m) = (Tn)^m$.
  Similarly, the underlying strict finite-power-preserving functor of
  $\mathbb{F}^\mathrm{op} \rightarrow \bar \T$ is given by
  postcomposition with the unit of $T$, and so is precisely the Lawvere
  theory corresponding to the finitary monad $T$.
\end{exa}

To reconstruct the whole pseudocommutative triangle in~\eqref{eq:8},
we need the following result, which combines Propositions~2.5 and
4.4~of~\cite{Garner2014Lawvere}; we omit the proof for now,\pagebreak~ though
note that the corresponding generalisations over a general locally
finitely presentable base will be proven as Propositions~\ref{prop:10}
and~\ref{prop:32} below.

\begin{prop}
  \label{prop:16}
  Let $\S$ denote the $\F$-category corresponding to the
  category-with-finite-powers $\cat{Set}$. The embeddings of finitary
  monads and Lawvere theories into $\F$-categories fit into
  pseudocommutative triangles:
  \begin{equation*}
    \cd[@!C@C-4em]{
      \cat{Mnd}_f(\cat{Set})^\mathrm{op}
      \ar[dr]_-{\cat{Alg}(\thg)}\ar[rr]^(0.5){} & \twosim{d} & \F\text-\cat{CAT}^\mathrm{op}
      \ar[dl]^-{\F\text-\cat{CAT}(\thg, \S)} \\ & \cat{CAT} &
    } \qquad \text{and} \qquad 
    \cd[@!C@C-3em]{
      \cat{Law}^\mathrm{op}
      \ar[dr]_-{\cat{Mod}(\thg)}\ar[rr]^(0.5){} & \twosim{d} & \F\text-\cat{CAT}^\mathrm{op}
      \ar[dl]^-{\F\text-\cat{CAT}(\thg, \S)}\rlap{ .} \\ & \cat{CAT} &
    }
  \end{equation*}
\end{prop}

Given this, to obtain the compability with semantics in~\eqref{eq:8},
it suffices to show that, for any one-object $\F$-category $\C$ with
completion under absolute tensors $\bar \C$, there is an equivalence
between the category of $\F$-functors $\C \rightarrow \S$ and the
category of $\F$-functors $\bar \C \rightarrow \S$. But since by
construction $\S$ is absolute-tensored, this follows directly from the
universal property of free completion under absolute tensors.

\section{Ingredients for generalisation}
\label{sec:ingr-gener}

In the rest of the paper, we extend the analysis of the previous
section to deal with the finitary monad--Lawvere correspondence over
an arbitrary locally finitely presentable (lfp) base. In this section,
we set up the necessary background for this: first recalling
from~\cite{Nishizawa2009Lawvere} the details of the generalised
finitary monad--Lawvere theory correspondence, and then recalling
from~\cite{Walters1981Sheaves, Street1983Enriched} some necessary
aspects of bicategory-enriched category theory. We will assume
familiarity with the basic theory of lfp categories as found, for
example, in~\cite{Adamek1994Locally,Gabriel1971Lokal}.

\subsection{The monad--theory correspondence over a general lfp base}
\label{sec:lawvere-a-theories}

In extending the monad--theory correspondence~\eqref{eq:8} from
$\cat{Set}$ to a given lfp category $\A$, one side of the
generalisation is apparent: we simply replace finitary monads on
$\cat{Set}$ by finitary monads on $\A$. On the other side, the
appropriate generalisation of Lawvere theories is given by the
\emph{Lawvere $\A$-theories} of~\cite{Nishizawa2009Lawvere}:

\begin{defi}
  \label{def:15}
  A \emph{Lawvere $\A$-theory} is a small category $\L$ together with
  an identity-on-objects, finite-limit-preserving functor
  $J \colon \af \rightarrow \L$. A \emph{morphism} of Lawvere
  $\A$-theories is a functor $\L \rightarrow \L'$ commuting with the
  maps from $\af$. A \emph{model} for a Lawvere $\A$-theory is a
  functor $F \colon \L \rightarrow \cat{Set}$ for which $FJ$ preserves
  finite limits.
\end{defi}

Here, and in what follows, we write $\A_f$ for a small subcategory
equivalent to the full subcategory of finitely presentable objects in
$\A$. Note that, while $\af$ has all finite limits, we do not assume
the same of $\L$; though, of course, it will admit all finite limits
of diagrams in the image of $J$, and so in particular all finite
products.

\begin{rem}
  \label{rk:2}
  Our definition of Lawvere $\A$-theory alters that
  of~\cite{Lack2009Gabriel-Ulmer,Nishizawa2009Lawvere} by dropping the
  requirement of \emph{strict} finite limit preservation. However,
  this apparent relaxation does not in fact change the notion of
  theory. To see why, we must consider carefully what this strictness
  amounts to, which is delicate, since $\L$ need not be finitely
  complete. The correct interpretion is as follows: we fix some choice
  of finite limits in $\af$, and also assume that, for each finite
  diagram $D \colon \I \rightarrow \af$, the category $\L$ is endowed
  with a choice of limit for $JD$ (in particular, because $J$ is the
  identity on objects, this equips $\L$ with a choice of finite
  products). We now require that $J \colon \af \rightarrow \L$ send
  the chosen limits in $\af$ to the chosen limits in $\L$. However, in
  this situation, the choice of limits in $\L$ is uniquely determined
  by that in $\af$ \emph{so long as} $J$ preserves finite limits in
  the non-algebraic sense of sending limit cones to limit cones. Thus,
  if we interpret the preservation of finite limits in
  Definition~\ref{def:15} in this non-algebraic sense, then our notion
  of Lawvere $\A$-theory agrees
  with~\cite{Lack2009Gabriel-Ulmer,Nishizawa2009Lawvere}.

  Note also that our definition of model for a Lawvere $\A$-theory is
  that of~\cite{Lack2009Gabriel-Ulmer}, rather than that
  of~\cite{Nishizawa2009Lawvere}: the latter paper defines a model to
  comprise $A \in \A$ together with a functor
  $F \colon \L \rightarrow \cat{Set}$ such that
  $FJ = \A(\thg, A) \colon \af \rightarrow \cat{Set}$. The equivalence
  of these definitions follows since $\A$ is equivalent to the
  category $\cat{FL}({\A_f}^\mathrm{op}, \cat{Set})$ of finite-limit-preserving functors
  $\af \rightarrow \cat{Set}$ via the assignation
  $A \mapsto \A(\thg, A)$.
\end{rem}

Even bearing the above remark in mind, it is not immediate that
Lawvere $\cat{Set}$-theories and their models coincide with Lawvere
theories and their models in the previous sense; however, this was
shown to be so in~\cite[Theorem~2.4]{Nishizawa2009Lawvere}. The
correctness of these notions over a general base is confirmed by the
main result of~\cite{Nishizawa2009Lawvere}, which we restate here as:

\begin{thm}\label{thm:4}
  \cite[Corollary 5.2]{Nishizawa2009Lawvere}
  The category of finitary monads on $\A$ is equivalent to the
  category of Lawvere $\A$-theories; moreover this equivalence is
  compatible with the semantics in the sense displayed in~\eqref{eq:7}.
\end{thm}

\begin{proof}[Proof (sketch)]
  For a finitary monad $T$ on $\A$, let $\L_T$ be the category with
  objects those of ${\A_f}$, hom-sets $\L_T(A,B) = \A(B,TA)$, and the
  usual Kleisli composition; now the identity-on-objects
  $J_T \colon {\A_f}^\mathrm{op} \rightarrow \L_T$ sending
  $f \in \A_f(B,A)$ to $\eta_A \circ f \in \L_T(A,B)$ is a Lawvere
  $\A$-theory. Conversely, if
  $J \colon {\A_f}^\mathrm{op} \rightarrow \L$ is a Lawvere
  $\A$-theory, then the composite of the evident forgetful functor
  $\cat{Mod}(\L) \rightarrow \cat{FL}({\A_f}^\mathrm{op}, \cat{Set})$
  with the equivalence
  $\cat{FL}({\A_f}^\mathrm{op}, \cat{Set}) \rightarrow \A$ is
  finitarily monadic, and so gives a finitary monad $T_\L$ on $\A$.
  With some care one may now show that these processes are
  pseudoinverse in a manner which is compatible with the semantics.
\end{proof}

\subsection{Bicategory-enriched category theory}
\label{sec:bicat-enrich-categ}

We now recall some basic definitions from the theory of categories
enriched over a bicategory as developed in~\cite{Walters1981Sheaves,
  Street1983Enriched}.

\begin{defi}
  \label{def:1}
  Let $\W$ be a bicategory whose $1$-cell composition and identities
  we write as $\comp$ and $\unit_x$ respectively. A
  \emph{$\W$-category} $\C$ comprises:
  \begin{itemize}
  \item A set $\ob \C$ of objects;
  \item For each $X \in \ob \C$, an \emph{extent}
    $\epsilon X \in \ob \W$;
  \item For each $X, Y \in \ob \C$, a hom-object\footnote{There are
      two conventions in the literature: we either take $\C(X,Y) \in
      \W(\epsilon X, \epsilon Y)$, as in~\cite{Gordon1997Enrichment}
      for example, or we take $\C(X,Y) \in
      \W(\epsilon Y, \epsilon X)$ as in~\cite{Street1983Enriched}. We
      have chosen the former convention here, and have 
      adjusted results from the literature where necessary to conform
      with this.}
    $\C(X,Y) \in \W(\epsilon X, \epsilon Y)$;
  \item For each $X, Y, Z \in \ob \C$, composition maps in
    $\W(\epsilon X, \epsilon Z)$ of the form
    \begin{equation*}
      \mu_{xyz} \colon \C(Y,Z) \comp \C(X,Y) \rightarrow \C(X,Z) \rlap{ ;}
    \end{equation*}
  \item For each $X \in \ob \C$ an identities map in
    $\W(\epsilon X, \epsilon X)$ of the form
    \begin{equation*}
      \iota_X \colon \unit_{\epsilon X} \rightarrow \C(X,X) \rlap{ ;}
    \end{equation*}
  \end{itemize}
  subject to associativity and unitality axioms. A \emph{$\W$-functor}
  $F \colon \C \rightarrow \D$ between $\W$-categories comprises an
  extent-preserving assignation on objects, together with maps
  $F_{XY} \colon \C(X,Y) \rightarrow \D(FX,FY)$ in
  $\W(\epsilon X, \epsilon Y)$ for each $X,Y \in \ob \C$, subject to the
  two usual functoriality axioms. Finally, a \emph{$\W$-transformation}
  $\alpha \colon F \Rightarrow G \colon \C \rightarrow \D$ between
  $\W$-functors comprises maps
  $\alpha_x \colon \unit_{\epsilon X} \rightarrow \D(FX,GX)$ in
  $\W(\epsilon X, \epsilon X)$ for each $X \in \ob \C$ obeying a
  naturality axiom. We write $\W\text-\cat{CAT}$ for the $2$-category of
  $\W$-categories, $\W$-functors and $\W$-transformations.
\end{defi}

Note that, if $\W$ is the one-object bicategory corresponding to a
monoidal category $\V$, then we re-find the usual definitions of
$\V$-category, $\V$-functor and $\V$-transformation. A key difference
in the general bicategorical situation is that a $\W$-category does
not have a single ``underlying ordinary category'', but a whole family
of them:

\begin{defi}
  \label{def:7}
  For any $x \in \W$, we write $\I_x$ for the $\W$-category with a
  single object $\ast$ of extent $x$ and with
  $\I_A(\ast, \ast) = \unit_x$, and write $(\thg)_x$ for the
  representable $2$-functor
  $\W\text-\cat{CAT}(\I_x, \thg) \colon \W\text-\cat{CAT} \rightarrow
  \cat{CAT}$. On objects, this $2$-functor sends a $\W$-category $\C$
  to the ordinary category $\C_x$ whose objects are the objects of $\C$ with
  extent $x$, and whose morphisms $X \rightarrow Y$ are morphisms
  $\unit_x \rightarrow \C(X,Y)$ in $\W(x,x)$.
\end{defi}

\subsection{Enrichment through variation}
\label{sec:enrichm-thro-vari}

It was shown in~\cite{Gordon1997Enrichment} that there is a close link
between $\W$-categories and $\W$-\emph{representations}. A
$\W$-representation is simply a homomorphism of bicategories
$F \colon \W \rightarrow \cat{CAT}$, but thought of as a ``left
action''; thus, we notate the functors
$F_{xy} \colon \W(x,y) \rightarrow \cat{CAT}(Fx,Fy)$ as
$W \mapsto W \ast_F (\thg)$, and write the components of the coherence
isomorphisms for $F$ as maps
$\lambda \colon I_x \ast_F X \rightarrow X$ and
$\alpha \colon (V \otimes W) \ast_F X \rightarrow V \ast_F (W \ast_F
X)$.

Theorem~3.7 of~\cite{Gordon1997Enrichment} establishes an equivalence
between \emph{closed} $\W$-representations and \emph{tensored}
$\W$-categories. Here, a $\W$-representation is called \emph{closed}
if each functor $(\thg) \ast_F W \colon \W(x,y) \rightarrow Fy$ has a
right adjoint $\spn{W, \thg}_F \colon Fy \rightarrow \W(x,y)$; while a
$\W$-category is \emph{tensored} if it admits all tensors in the
following sense:

\begin{defi}
  If $\W$ is a bicategory and $\C$ is a $\W$-category, then a
  \emph{tensor} of $X \in \C_x$ by $W \in \W(x,y)$ is an object
  $W \cdot X \in \C_y$ together with a map
  $u \colon W \rightarrow \C(X, W \cdot X)$ in $\W(x,y)$ such that,
  for any $U \in \W(y,z)$ and $Z \in \C_z$, the assignation
  \begin{equation}\label{eq:38}
    U \xrightarrow{f} \C(W \cdot X, Z) \ \ \mapsto \ \ 
    U \otimes W \xrightarrow{f \otimes u} \C(V \cdot X, Z) \otimes \C(X,
    W \cdot X) \xrightarrow{\circ} \C(X, Z)
  \end{equation}
  establishes a bijection between morphisms $U \rightarrow \C(W \cdot X, Z)$
  in $\W(y,z)$ and morphisms $U \otimes W \rightarrow \C(X,Z)$ in
  $\W(x,z)$.
\end{defi}

We note here for future use that a tensor $W \cdot X$ is said to be
\emph{preserved} by a $\W$-functor $F \colon \C \rightarrow \D$ if the
composite
$W \rightarrow \C(X, W \cdot X) \rightarrow \D(FX, F(W \cdot X))$
exhibits $F(W \cdot X)$ as $W \cdot FX$; and that tensors by a
$1$-cell $W$ are called \emph{absolute} if they are preserved by any
$\W$-functor.
 \begin{prop}
   \label{prop:29}\cite[Theorem~3.7]{Gordon1997Enrichment}
   There is an equivalence between the $2$-category
   $\W\text-\cat{CAT}_{\mathrm{tens}}$ of tensored $\W$-categories, 
   tensor-preserving $\W$-functors and $\W$-natural transformations, and the
   $2$-category $\mathrm{Hom}(\W, \cat{CAT})_\mathrm{cl}$ of closed
   $\W$-representations, pseudonatural transformations and modifications.
 \end{prop}

 \begin{proof}[Proof (sketch)]
   In one direction, the closed $\W$-representation $C$ associated to a
   tensored $\W$-category $\C$ is defined on objects by $C(x) = \C_x$,
   and with action by $1$-cells given by tensors:
   $W \ast_{C} X = W \cdot X$. We do not need the further details here,
   and so omit them. In the other direction, the tensored $\W$-category
   $\F$ associated to a closed representation
   $F \colon \W \rightarrow \cat{CAT}$ has objects of extent $a$ being
   objects of $Fa$; hom-objects given by $\F(X,Y) = \spn{X,Y}_F$; and
   composition and identities given by transposing the maps
   \begin{equation*}
     (\spn{Y,Z}_F \otimes \spn{X,Y}_F) \ast_F X \xrightarrow{\alpha}
     \spn{Y,Z}_F \ast_F (\spn{X,Y}_F \ast_F X) \xrightarrow{1 \ast \varepsilon}
     \spn{Y,Z}_F \ast_F Y
     \xrightarrow{\varepsilon} Z
   \end{equation*}
   and $\lambda \colon I_a \ast_F X \rightarrow X$ under the closure
   adjunctions. The $\W$-category $\F$ so obtained admits all tensors on
   taking $W \cdot X = W \ast_F X$ with unit
   $W \rightarrow \spn{X, W \ast_F X}_F$ obtained from the closure
   adjunctions.
 \end{proof}

We will make use of this equivalence in
Section~\ref{sec:algebr-finit-monads} below, and will require the
following easy consequence of the definitions:

\begin{prop}
  \label{prop:9}
  Let $F \colon \W \rightarrow \cat{CAT}$ be a closed representation,
  corresponding to the tensored $\W$-category $\F$, and let
  $T \colon a \rightarrow a$ be a monad in $\W$, corresponding to the
  one-object $\W$-category $\T$. There is an isomorphism of categories
  $\W\text-\cat{CAT}(\T, \F) \cong F(T)\text-\cat{Alg}$, natural in
  maps of monads on $a$ in $\W$.
\end{prop}

\begin{proof}
  The action on objects of a $\W$-functor $\T \to \F$ picks out an
  object of extent $a$ in $\F$, thus, an object $X \in Fa$. The action
  on homs is given by a map $x \colon T \to \spn{X,X}_F$ in $\W(a,a)$,
  while functoriality requires the commutativity of:
  \begin{equation*}
    \cd[@!C@C-1.5em@-0.5em]{
      & I_a \ar[dl]_\eta \ar[dr]^\iota \\
      T \ar[rr]^x & & \spn{X,X}_F} \qquad \text{and} \qquad
    \cd[@C+0.5em@-0.5em]{ T \otimes T \ar[r]^-{x \otimes x}
      \ar[d]_{\mu} &
      \spn{X,X}_F \otimes \spn{X,X}_F \ar[d]^{\mu} \\
      T \ar[r]^-x & \spn{X,X}_F\rlap{ .} }
  \end{equation*}
  Transposing under adjunction, this is equally to give $X \in Fa$ and a map
  $T \ast_F X \to X$ satisfying the two axioms to be an algebra for
  $F(T) = T \ast_F (\thg)$. Further, to give a $\W$-transformation
  $F \Rightarrow G \colon \T \to \F$ is equally to give
  $\phi \colon I_x \to \spn{X,Y}_F$ such that
  \begin{equation*}
    \cd[@C+1em]{
      T \ar[r]^{y} \ar[d]_{x} & \spn{Y, Y}_F \ar[r]^-{\varphi \otimes 1}
      & \spn{X,Y}_F \otimes \spn{Y,Y}_F \ar[d]^{\mu} \\
      \spn{X, X}_F \ar[r]^-{1 \otimes \varphi} &
      \spn{X,X}_F\otimes\spn{X,Y}_F \ar[r]^-{\mu} & \spn{X, Y}_F }
  \end{equation*}
  commutes; which, transposing under adjunction and using the coherence
  constraint $I_a \ast X \cong X$, is equally to give a map $X \to Y$
  commuting with the $F(T)$-actions. The naturality of the
  correspondence just described in $T$ is easily checked.
\end{proof}

\section{Finitary monads and their algebras}
\label{sec:lex-prof-finit}

We now begin our enriched-categorical analysis of the monad--theory
correspondence over an lfp base. We first describe a bicategory $\lp$
of \emph{lex profunctors} which is biequivalent to the bicategory
$\lfp$ of finitary functors between lfp categories, but more
convenient to work with; we then exhibit each finitary monad on an lfp
category as a $\lp$-category, and the associated category of algebras
as a category of $\lp$-enriched functors.

\subsection{Finitary monads as enriched categories}
\label{sec:lex-profunctors}

The basic theory of lfp categories tells us that for any small
finitely-complete $\mathbb{A}$, the category
$\fl(\mathbb{A}, \cat{Set})$ of finite-limit-preserving functors
$\mathbb{A} \rightarrow \cat{Set}$ is lfp, and moreover that every lfp
category is equivalent to one of this form. So $\lfp$ is biequivalent
to the bicategory whose objects are small finitely-complete
categories, whose hom-category from $\mathbb{A}$ to $\mathbb{B}$ is
$\lfp(\fl(\mathbb{A}, \cat{Set}), \fl(\mathbb{B}, \cat{Set}))$, and
whose composition is inherited from~$\lfp$.

Now, since for any small finitely-complete $\mathbb{A}$, the inclusion
$\mathbb{A}^\mathrm{op} \rightarrow \fl(\mathbb{A}, \cat{Set})$
exhibits its codomain as the free filtered-cocomplete category on its
domain, there are equivalences
$\lfp(\fl(\mathbb{A}, \cat{Set}), \fl(\mathbb{B}, \cat{Set})) \simeq
[\mathbb{A}^\mathrm{op}, \fl(\mathbb{B}, \cat{Set})]$; thus
transporting the compositional structure of $\lfp$ across these
equivalences, we obtain:

\begin{defi}
  \label{def:13}
  The right-closed bicategory $\lp$ of \emph{lex profunctors} has:
  \begin{itemize}

  \item As objects, small categories with finite limits.

  \item
    $\lp(\mathbb{A}, \mathbb{B}) = [\mathbb A^\mathrm{op},
    \fl(\mathbb{B}, \cat{Set})]$; we typically identify objects therein
    with functors
    $\mathbb{A}^\mathrm{op} \times \mathbb{B} \rightarrow \cat{Set}$
    that preserve finite limits in their second variable.

  \item The identity $1$-cell
    $I_\mathbb{A} \in \lp(\mathbb{A}, \mathbb{A})$ is given by
    $I_\mathbb{A}(a',a) = \mathbb{A}(a',a)$, while the composition of
    $M \in \lp(\mathbb{A}, \mathbb{B})$ and
    $N \in \lp(\mathbb{B}, \mathbb{C})$ is given by:
    \begin{equation}\label{eq:15}
      (N \otimes M)(a, c) = \textstyle\int^{b \in \mathbb{B}} N(b,c) \times M(a,b)\rlap{ .}
    \end{equation}

  \item For $M \in \lp(\mathbb{A}, \mathbb{B})$ and
    $P \in \lp(\mathbb{A}, \mathbb{C})$, the right closure
    $[M,P] \in \lp(\mathbb{B}, \mathbb{C})$ is defined by
    \begin{equation}
      \label{eq:41}
      [M,P](b,c) = \textstyle\int_a\, [M(a,b), P(a,c)]\rlap{ .}
    \end{equation}
  \end{itemize}
\end{defi}

By the above discussion, $\lfp$ is biequivalent to $\lp$, and this
induces an equivalence between the category of monads on $\A$ in
$\lfp$---thus, the category of finitary monads on $\A$---and the
category of monads on $\af$ in $\lp$. Such monads correspond with
one-object $\lp$-categories of extent $\af$ and so:

\begin{prop}
  \label{prop:8}
  For any locally finitely presentable category $\A$, the category
  $\cat{Mnd}_f(\A)$ of finitary monads on $\A$ is equivalent to the
  category of $\lp$-categories with a single object of extent $\af$.
\end{prop}

\subsection{Algebras for finitary monads as enriched functors}
\label{sec:algebr-finit-monads}

We now explain algebras for finitary monads in the $\lp$-enriched
context. Composing the biequivalence $\lp \rightarrow \lfp$ with the
inclusion $2$-functor $\lfp \rightarrow \cat{CAT}$ yields a
homomorphism $S \colon \lp \rightarrow \cat{CAT}$ which on objects
sends $\mathbb{A}$ to $\fl(\mathbb{A}, \cat{Set})$, and for which the
action of a $1$-cell $M \in \lp(\mathbb{A}, \mathbb{B})$ on an object
$X \in \fl(\mathbb{A}, \cat{Set})$ is given as on the left in
\begin{equation*}
  (M \ast_S X)(b) = \textstyle\int^{a \in \mathbb{A}} M(a,b) \times Xa \qquad\qquad 
  \spn{X,Y}_S(a,b) = \cat{Set}(Xa,Yb)\rlap{ .}
\end{equation*}
This $S$ is a closed representation, where for
$X \in \fl(\mathbb{A}, \cat{Set})$ and
$Y \in \fl(\mathbb{B}, \cat{Set})$ we define $\spn{X,Y}_S$ as to the
right above; and so applying Proposition~\ref{prop:29} gives a
tensored $\lp$-category $\S$ with objects of extent $\mathbb{A}$ being
finite-limit-preserving functors $\mathbb{A} \rightarrow \cat{Set}$,
and with hom-objects $\S(X,Y)(a,b) = \cat{Set}(Xa, Yb)$.

\begin{prop}
  \label{prop:10}
  For any locally finitely presentable category $\A$, the embedding of finitary
  monads on $\A$ as one-object $\lp$-categories obtained in
  Proposition~\ref{prop:8} fits into a triangle, commuting up to
  pseudonatural equivalence:
  \begin{equation}\label{eq:10}
    \cd[@!C@C-4em]{
      \cat{Mnd}_f(\A)^\mathrm{op}
      \ar[dr]_-{(\thg)\text-\cat{Alg}}\ar[rr]^(0.5){} &
      \twosim{d} & (\lp\text-\cat{CAT})^\mathrm{op}\rlap{ .}
      \ar[dl]^-{\ \ \ \lp\text-\cat{CAT}(\thg, \S)} \\ & \cat{CAT} &
    }
  \end{equation}
\end{prop}

\begin{proof}
  Given $T \in \cat{Mnd}_f(\A)$, which is equally a monad on $\A$ in
  $\lfp$, we can successively apply the biequivalences
  $\lfp \rightarrow \lp$ and $\lp \rightarrow \lfp$ to
  obtain in turn a monad $T'$ on $\A_f^\mathrm{op} \in \lp$ and a
  monad $T''$ on $\cat{FL}(\A_f^\mathrm{op}, \cat{Set})$. It follows
  easily from the fact of a biequivalence that
  $T\text-\cat{Alg} \simeq T''\text-\cat{Alg}$.

  Now, starting from $T \in \cat{Mnd}_f(\A)$, the functor across the
  top of~\eqref{eq:10} sends it to the one-object $\lfp$-category
  $\T'$ corresponding to $T'$; whereupon by Proposition~\ref{prop:9},
  we have pseudonatural equivalences
  \begin{equation*}
    \lp\text-\cat{CAT}(\T', \S) \cong S(T')\text-\cat{Alg} =
    T''\text-\cat{Alg} \simeq T\text-\cat{Alg}\text{ .} \tag*{\qEd}
  \end{equation*}
  \def\popQED{}
\end{proof}


\subsection{General $\lp$-categories}
\label{sec:notation}

Before turning to the relationship of $\lp$-categories and Lawvere
$\A$-theories, we take a moment to unpack the data for a general
$\lp$-category $\C$. We have objects $X, Y, \dots$ with associated
extents $\mathbb{A}, \mathbb{B}, \dots$ in $\lp$; while for objects
$X \in \C_\mathbb{A}$ and $Y \in \C_\mathbb{B}$, we have the
hom-object
$\C(X,Y) \colon \mathbb{A}^\mathrm{op} \times \mathbb{B} \rightarrow
\cat{Set}$, which is a functor preserving finite limits in its second
variable. By the coend formula~\eqref{eq:15} for $1$-cell composition
in $\lp$, composition in $\C$ is equally given by functions
\begin{equation}\label{eq:29}
  \begin{aligned}
    \C(Y,Z)(j,k) \times \C(X,Y)(i,j) &\rightarrow \C(X,Z)(i, k) \\
    (g,f) & \mapsto g \circ f
  \end{aligned}
\end{equation}
which are natural in $i \in \mathbb{A}$ and $k \in \mathbb{C}$ and
dinatural in $j \in \mathbb{B}$. On the other hand, identities in $\C$
are given by functions
$\iota_X \colon \mathbb{A}(i,j) \rightarrow \C(X,X)(i,j)$, natural in
$i,j \in \mathbb{A}$; if we define $1_{X,i} \defeq \iota_X(1_i)$, then
the $\lp$-category axioms for $\C$ say that
$f \circ 1_{X,i} = f = 1_{Y,j} \circ f$ for all $f \in \C(X,Y)(i,j)$
and that the operation~\eqref{eq:29} is associative. Note that the
naturality of each $\iota_X$ together with the unit axioms imply that
the action on morphisms of the hom-object $\C(X,Y)$ is given by
\begin{equation}\label{eq:5}
  \begin{aligned}
    \C(X,Y)(\varphi, \psi) \colon \C(X,Y)(i,j) &\rightarrow
    \C(X,Y)(i',j')\\
    f & \mapsto \iota_Y(\psi) \circ f \circ \iota_X(\varphi)\rlap{ .}
  \end{aligned}
\end{equation}
Applying naturality of $\iota_X$ again to this formula yields the
following functoriality equation for any pair of composable maps in
$\mathbb{A}$:
\begin{equation}
  \label{eq:30}
  \iota_X(\varphi' \circ \varphi) = \iota_X(\varphi') \circ
\iota_X(\varphi)\rlap{ .}
\end{equation}

\section{Partial finite completeness}
\label{sec:part-finite-completeness}

In the following two sections, we will identify the absolute-tensored
$\lp$-categories with what we call \emph{partially finitely complete}
ordinary categories; this identification will take the form of a
biequivalence between suitably-defined $2$-categories. We will exploit
this biequivalence in Section~\ref{sec:lawv-lp-categ} in order to
identify Lawvere $\A$-theories with certain functors between absolute-tensored
$\lp$-categories.

\subsection{Partially finitely complete categories}
\label{sec:part-finit-compl}
We begin by introducing the $2$-category of partially finitely
complete categories and partially finite-limit-preserving functors.
\begin{defi}
  \label{def:12}
  By a \emph{left-exact sieve} on a category $\C$, we mean a
  collection $\S$ of finite-limit-preserving functors
  $\mathbb{A} \rightarrow \C$, each with small, finitely-complete
  domain, and satisfying the following conditions, wherein we write
  $\S[\mathbb{A}]$ for those elements of $\S$ with domain
  $\mathbb{A}$:
  \begin{enumerate}[(i)]
  \item If $X \in \S[\mathbb{B}]$ and
    $G \in \cat{FL}(\mathbb{A}, \mathbb{B})$, then
    $XG \in \S[\mathbb{A}]$;
  \item If $X \in \S[\mathbb{A}]$ and
    $X \cong Y \colon \mathbb{A} \rightarrow \C$, then
    $Y \in \S[\mathbb{A}]$;
  \item Each object of $\C$ is in the image of some functor in $\S$.
  \end{enumerate}
  A \emph{partially finitely complete category} $(\C, \S_\C)$ is a
  category $\C$ together with a left-exact sieve $\S_\C$ on it. Where
  confusion is unlikely, we may write $(\C, \S_\C)$ simply as $\C$. A
  \emph{partially finite-limit-preserving functor}
  $(\C, \S_\C) \rightarrow (\D, \S_\D)$ is a functor
  $F \colon \C \rightarrow \D$ such that $FX \in \S_\D$ for all
  $X \in \S_\C$; we call such an $F$ \emph{sieve-reflecting} if, for
  all $Y \in \S_\D$, there exists $X \in \S_\C$ such that
  $FX \cong Y$. We write $\cat{PARFL}$ for the $2$-category of
  partially finitely complete categories, partially finite-limit-preserving functors,
  and arbitrary natural transformations.
\end{defi}

The following examples should serve to clarify the relevance of these
notions to Lawvere theories over a general lfp base.

\begin{exa}
  \label{ex:1}
  Any finitely complete $\C$ can be seen as partially finitely
  complete when endowed with the sieve $\S_\C$ of all
  finite-limit-preserving functors into $\C$ with small domain. If
  $\D$ is also finitely complete, then any finite-limit-preserving
  $F \colon \C \rightarrow \D$ is clearly also partially finite-limit-preserving;
  conversely, if $F \colon \C \rightarrow \D$ is partially finite-limit-preserving,
  then for any finite diagram $D \colon \mathbb{I} \rightarrow \C$,
  closing its image in $\C$ under finite limits yields a small
  subcategory $\mathbb{A}$ for which the full inclusion
  $J \colon \mathbb{A} \rightarrow \C$ preserves finite limits. As $F$
  is partially finite-limit-preserving, the composite
  $FJ \colon \mathbb{A} \rightarrow \D$ also preserves finite limits;
  in particular, the chosen limit cone over $D$ in $\C$---which lies
  in the subcategory $\mathbb{A}$---is sent to a limit cone in $\D$.
  It follows there is a full and locally full inclusion of
  $2$-categories $\fl \rightarrow \pfl$.
\end{exa}

\begin{exa}
  \label{ex:2}
  If $J \colon \A_f^\mathrm{op} \rightarrow \L$ is a Lawvere
  $\A$-theory, then $\L$ becomes a partially finitely complete
  category when endowed with the sieve generated by $J$:
  \begin{equation}\label{eq:9}
    \S_\L = \{ F \colon \mathbb{A} \rightarrow \L : F \cong JG \text{ for some
      finite-limit-preserving }
    G \colon \mathbb{A} \rightarrow \A_f^\mathrm{op}\}\rlap{ .}
  \end{equation}
  Clearly $\S_\L$ satisfies conditions (i) and (ii) above, and
  satisfies (iii) by virtue of $J$ being bijective on objects.
  Moreover, a partially finite-limit-preserving $\L \rightarrow \C$ is precisely a
  functor $F \colon \L \rightarrow \C$ such that $FJ \in \S_\C$; so in
  particular, a partially finite-limit-preserving $\L \rightarrow \cat{Set}$ is
  precisely a model for the Lawvere $\A$-theory $\L$.
\end{exa}

\subsection{Partial finite completeness and $\lp$-enrichment}

Towards our identification of absolute-tensored $\lp$-categories with
partially finitely complete categories, we now construct a
$2$-adjunction
\begin{equation}\label{eq:34}
  \cd{
    {\pfl} \ar@<-4.5pt>[r]_-{\Gamma} \ar@{}[r]|-{\bot} &
    {\lp\text-\cat{CAT}} \ar@<-4.5pt>[l]_-{\int}\rlap{ .}
  }
\end{equation}

\begin{defi}
  \label{def:17}
  Let $\C$ be a partially finitely complete category. The
  $\lp$-category $\Gamma(\C)$ has objects of extent $\mathbb{A}$ given
  by elements $X \in \S_\C[\mathbb{A}]$, and remaining data defined as
  follows:
  \begin{itemize}[itemsep=0.25\baselineskip]
  \item For $X \in \S_\C[\mathbb{A}]$ and $Y \in \S_\C[\mathbb{B}]$,
    the hom-object $\Gamma(\C)(X,Y) \in \lp(\mathbb{A}, \mathbb{B})$
    is given by $\Gamma(\C)(X,Y)(i,j) = \C(Xi,Yj)$. Note that this
    preserves finite limits in its second variable since $Y$ and each
    $\C(Xi, \thg)$ do so.
  \item Composition in $\Gamma(\C)$ may be specified, as
    in~\eqref{eq:29}, by natural families of functions
    $\Gamma(\C)(Y,Z)(j,k) \times \Gamma(\C)(X,Y)(i,j) \rightarrow
    \Gamma(\C)(X,Z)(i,k)$, which we obtain from composition in $\C$.
  \item Identities
    $\iota_X \colon \mathbb{A}(i,j) \rightarrow \Gamma(\C)(X,X)(i,j) =
    \C(Xi,Xj)$ are given by the action of $X$ on morphisms.
  \end{itemize}
  The $\lp$-category axioms for $\Gamma(\C)$ follow from the category
  axioms of $\C$ and functoriality of each $X$.

  If $F \colon \C \rightarrow \D$ is a partially finite-limit-preserving functor,
  then we define the $\lp$-functor
  $\Gamma(F) \colon \Gamma(\C) \rightarrow \Gamma(\D)$ to have action
  on objects $X \mapsto FX$ (using the fact that $FX \in \S_\D$
  whenever $X \in \S_\C$). The components of the action of $\Gamma(F)$ on hom-objects
  $\Gamma(\C)(X,Y) \rightarrow \Gamma(\D)(FX,FY)$ are functions
  $\C(Xi,Yj) \rightarrow \D(FXi,FYj)$, which are given simply by the
  action of $F$ on morphisms. The $\lp$-functor axioms are immediate
  from functoriality of $F$.

  Finally, for a $2$-cell
  $\alpha \colon F \Rightarrow G \colon \C \rightarrow \D$ in
  $\pfl$, we define a $\lp$-transformation
  $\Gamma(\alpha) \colon \Gamma(F) \Rightarrow \Gamma(G)$ whose
  component $I_\mathbb{A} \rightarrow \Gamma(\D)(FX,GX)$ is given by
  the dinatural family of elements $\alpha_{Xi} \in \D(FXi,GXi)$. The
  $\lp$-naturality of $\Gamma(\alpha)$ amounts to the condition that
  $Gf \circ \alpha_{Xi} = \alpha_{Yj} \circ Ff \colon FXi \rightarrow
  GYj$ for all $f \colon Xi \rightarrow Yj$ in $\C$; which is so by
  naturality of $\alpha$.
\end{defi}

\begin{prop}
  \label{prop:28}
  The data of Definition~\ref{def:17} comprise the action on $0$-, $1$-, and $2$-cells 
  of a $2$-functor $\Gamma \colon \pfl \rightarrow
  \lp\text-\cat{CAT}$, Moreover the $2$-functor $\Gamma$ admits a left $2$-adjoint
  $\int \colon \lp\text-\cat{CAT} \rightarrow \pfl$.
\end{prop}

\begin{proof}
  The $2$-functoriality of $\Gamma$ is easy to check, and so it
  remains to construct its left $2$-adjoint $\int$. Given a
  $\lp$-category $\C$, we write $\ic$ for the category with:
  \begin{itemize}
  \item \textbf{Objects} of the form $ (X,i)$ where
    $X \in \C_\mathbb{A}$ and $i \in \mathbb{A}$;
  \item \textbf{Morphisms} $f \colon (X,i) \rightarrow (Y,j)$ being
    elements $f \in \C(X,Y)(i,j)$;
  \item \textbf{Identities} given by the elements
    $1_{X,i} \in \C(X,X)(i,i)$;
  \item \textbf{Composition} mediated by the functions~\eqref{eq:29}.
  \end{itemize}
  Given $X \in \C_\mathbb{A}$, we write
  $\iota_X \colon \mathbb{A} \rightarrow \ic$ for the functor given by
  $i \mapsto (X,i)$ on objects and by the identities map
  $\iota_X \colon \mathbb{A}(i,i') \rightarrow \C(X,X)(i,i')$ of $\C$
  on morphisms; note this is functorial by~\eqref{eq:30}. By the
  definition of $\ic$ and~\eqref{eq:5}, we have that
  \begin{equation}
    \label{eq:31}
    \C(X,Y) = (\ic)(\iota_X(\thg), \iota_Y(\thg)) \colon \mathbb{A}^\mathrm{op}
    \times \mathbb{B} \rightarrow \cat{Set}\rlap{ ;}
  \end{equation}
  in particular, as each $\C(X,Y)$ preserves finite limits in its
  second variable, each
  functor $\ic( (X,i), \iota_Y(\thg)) \colon \mathbb{B} \rightarrow \cat{Set}$
  preserve finite limits, whence each
  $\iota_Y \colon \mathbb{B} \rightarrow \ic$ preserves finite limits.
  It follows that $\ic$ is partially finitely complete when endowed
  with the left-exact sieve
  \begin{equation}\label{eq:35}
    \S_{\int\!\C} = \{ \,G \colon \mathbb{A} \rightarrow \Gamma(\C) : G \cong
    \iota_Y F
    \text{ for some $Y \in \C_\mathbb{B}$ and $F \in
      \cat{FL}(\mathbb{A}, \mathbb{B})$}\,\}\rlap{ .}
  \end{equation}

  We now show that $\ic$ provides the value at $\C$ of a left
  $2$-adjoint to $\Gamma$; thus, we must exhibit isomorphisms of
  categories, $2$-natural in $\D \in \pfl$, of the form:
  \begin{equation}\label{eq:32}
    \pfl(\ic, \D) \cong
    \lp\text-\cat{CAT}(\C, \Gamma(\D))\rlap{ .}
  \end{equation}
  Now, to give a partially finite-limit-preserving functor
  $F \colon \ic \rightarrow \D$ is to give:
  \begin{itemize}
  \item For all $X \in \C_\mathbb{A}$ and $i \in \mathbb{A}$ an object
    $F(X,i) \in \D$; and
  \item For all $f \in \C(X,Y)(i,j)$, a map
    $Ff \colon F(X,i) \rightarrow F(Y,j)$ in $\D$,
  \end{itemize}
  functorially with respect to the composition~\eqref{eq:29} and
  composition in $\D$, and subject to the requirement that
  $F\iota_X \in \S_\D[\mathbb{A}]$ for all $X \in \C_\mathbb{A}$. On
  the other hand, to give a $\lp$-functor
  $G \colon \C \rightarrow \Gamma(\D)$ is to give:
  \begin{itemize}
  \item For all $X \in \C_\mathbb{A}$, a functor
    $GX \in \S_\D[\mathbb{A}]$; and
  \item For all $f \in \C(X,Y)(i,j)$, an element of
    $\Gamma(\D)(GX, GY) = \D((GX)i, (GY)j)$, i.e., a map
    $Gf \colon (GX)i \rightarrow (GY)j$ in $\D$,
  \end{itemize}
  subject to the same functoriality condition. Thus, given
  $F \colon \ic \rightarrow \D$, we may define
  $\smash{\bar F} \colon \C \rightarrow \Gamma(\D)$ by taking
  $\bar FX = F\iota_X$ (which is in $\S_\D[\mathbb{A}]$ by assumption)
  and $\smash{\bar Ff} = Ff$; the functoriality is clear. On the other
  hand, given $G \colon \C \rightarrow \Gamma(\D)$, we may define
  $\smash{\bar G} \colon \ic \rightarrow \D$ by taking
  $\bar G(X,i) = (GX)i$ and $\bar Gf = Gf$. Functoriality is again
  clear, but we need to check that
  $\bar G\iota_X \in \S_\D[\mathbb{A}]$ for all $X \in \C_\mathbb{A}$.
  In fact we show that $\bar G\iota_X = GX$, which is in
  $\S_\D[\mathbb{A}]$ by assumption. On objects,
  $\bar G\iota_X(i) = \bar G(X,i) = (GX)i$ as required. On morphisms,
  the compatibility of $G$ with identities in $\C$ and $\Gamma(\D)$
  gives a commuting triangle of sets and functions:
  \begin{equation*}
    \!\!\!\!\!\!\!\!\!\!\!\!\!\!\!\!\!\!\cd[@!C@C-6em]
    {
      & {\mathbb{A}(i,j)} \ar[dl]_-{\iota_X} \ar[dr]^-{\iota_{GX}} \\
      {\C(X,X)(i,j)} \ar[rr]^-{G} & &
      {\Gamma(\D)(GX, GX)(i,j) = \rlap{$ \D(GXi, GXj)$ .}}
    }
  \end{equation*}
  The left-hand path maps $\varphi \in \mathbb{A}(i,j)$ to
  $G\iota_X(\varphi) = \bar G\iota_X(\varphi)$; while by definition of
  $\Gamma(\D)$ the right-hand path maps $\varphi$ to $GX(\varphi)$;
  whence $\bar G\iota_X = GX$ as required. It is clear from the above
  calculations that the assignations $F \mapsto \bar F$ and
  $G \mapsto \bar G$ are mutually inverse, which establishes the
  bijection~\eqref{eq:32} on objects.

  To establish~\eqref{eq:32} on maps, let
  $F_1, F_2 \colon \ic \rightrightarrows \D$. The components of a
  $\lp$-transformation
  $\bar \alpha \colon \bar F_1 \Rightarrow \bar F_2 \colon \C
  \rightarrow \Gamma(\D)$ comprise natural families of functions
  $\bar \alpha_{Xij} \colon \mathbb{A}(i,j) \rightarrow \Gamma(\D)(\bar F_1X,
  \bar F_2X)(i,j) = \D(F_1\iota_X(i), F_2\iota_X(j))$
  satisfying $\lp$-naturality. By Yoneda, each $\bar \alpha_{Xij}$ is
  uniquely determined by elements
  $\bar \alpha_{X,i} = \bar \alpha_{Xii}(\id_i) \in \D(F_1(X,i),
  F_2(X,i))$ satisfying
  $F_2\iota_X(\varphi) \circ \bar \alpha_{X,i} = \bar \alpha_{X,j}
  \circ F_1\iota_X(\varphi)$ for all $\varphi \in \mathbb{A}(i,j)$;
  their $\lp$-naturality is now the requirement that the square
  \begin{equation*}
    \cd[@C+2em]{
      {\C(X,Y)(i,j)} \ar[r]^-{F_1} \ar[d]_{F_2} &
      {\D(F_1(X,i),F_1(Y,j))} \ar[d]^{\bar \alpha_{Y,j} \circ (\thg)} \\
      {\D(F_2(X,i),F_2(Y,j))} \ar[r]^-{(\thg) \circ \bar \alpha_{X,i}} &
      {\D(F_1(X,i),F_2(Y,j))}
    }
  \end{equation*}
  commute for each $X,Y,i,j$. Note that this \emph{implies} the
  earlier condition that
  $F_2\iota_X(\varphi) \circ \bar \alpha_{X,i} = \bar \alpha_{X,j}
  \circ F_1\iota_X(\varphi)$ on taking $X=Y$ and evaluating at
  $\iota_X(\varphi)$; now evaluating at a general element, we get the
  condition that
  $F_2f \circ \bar \alpha_{X,i} = \bar \alpha_{Y,j} \circ F_1f$ for
  all $f \colon (X,i) \rightarrow (Y,j)$ in $\ic$---which says
  precisely that we have a natural transformation
  $\bar \alpha \colon F_1 \Rightarrow F_2 \colon \ic \rightarrow \D$.
  This establishes the bijection~\eqref{eq:32} on morphisms; the
  $2$-naturality in $\D$ is left as an easy exercise for the reader.
\end{proof}

\section{Absolute-tensored $\lp$-categories}
\label{sec:analysis-unit-counit}

In this section, we prove the key technical result of this paper,
Theorem~\ref{thm:3}, which shows that the
$2$-adjunction~\eqref{eq:34} exhibits $\pfl$ as biequivalent
to the full sub-$2$-category of $\lp\text-\cat{CAT}$ on the
absolute-tensored $\lp$-categories.

\subsection{Absolute tensors in $\W$-categories}
\label{sec:absolute-tensors-w}

We begin by characterising absolute tensors in $\W$-categories for an
arbitrary right-closed bicategory $\W$. Here, \emph{right-closedness}
is the condition that, for every $1$-cell $W \in \W(x,y)$ and every
$z \in \W$, the functor
$(\thg) \otimes W \colon \W(y,z) \rightarrow \W(x,z)$ admits a right
adjoint $[W, \thg] \colon \W(x,z) \rightarrow \W(y,z)$. In this
setting, we will show that tensors by a $1$-cell $W$ are absolute if
and only if $W$ is a right adjoint in $\W$. If $\W$ were both left-
and right- closed, this would follow from the characterisation of
enriched absolute colimits given in~\cite{Street1983Absolute}, but in
the absence of left-closedness, we need a different proof. The first
step is the following, which is a special case
of~\cite[Theorem~1.2]{Garner2014Diagrammatic}:

\begin{prop}
  \label{prop:17}
  Let $\W$ be a bicategory, let $\C$ be a $\W$-category,
  let $X \in \C_x$ and let $W \in \W(x,y)$. If $W$ admits the left
  adjoint $W^\ast \in \W(y,x)$, then there is a bijective
  correspondence between data of the following forms:
  \begin{enumerate}[(a)]
  \item A map $u \colon W \to \C(X, Y)$ in $\W(x,y)$ exhibiting $Y$ as
    $W \cdot X$;
  \item Maps $u \colon W \to \C(X, Y)$ in $\W(x,y)$ and
    $u^\ast \colon W^\ast \to \C(Y, X)$ in $\W(y,x)$ rendering
    commutative the squares:
    \begin{equation}
      \cd{
        I_y \ar[r]^\eta \ar[d]_{\iota} & W \otimes W^\ast \ar[d]^{u
          \otimes u^\ast} &
        W^\ast \otimes W \ar[r]^{\varepsilon} \ar[d]_{u^\ast \otimes u} & I_x \ar[d]^{\iota} \\
        \C(Y,Y) \ar@{<-}[r]^-{\ \mu} & \C(X,Y) \otimes \C(Y,X) &
        \C(Y,X) \otimes \C(X,Y) \ar[r]^-{\mu} & \C(X,X)\rlap{ .}
      }\label{eq:4}
    \end{equation}
  \end{enumerate}
\end{prop}

\begin{proof}
  Given (a), applying surjectivity in~\eqref{eq:38} to
  $\iota_X \circ \varepsilon \colon W^\ast \otimes W \rightarrow I_x
  \rightarrow \C(X,X)$ yields a unique map
  $u^\ast \colon W^\ast \rightarrow \C(Y,X)$ making the square right
  above commute. To see that the left square also commutes, it
  suffices by injectivity in~\eqref{eq:38} to check that the sides become
  equal after tensoring on the right with $u$ and postcomposing with
  $\mu \colon \C(Y,Y) \otimes \C(X,Y) \rightarrow \C(X,Y)$. This
  follows by a short calculation using commutativity in the right
  square and the triangle identities.

  To complete the proof, it remains to show that if $u$ and $u^\ast$
  are given as in (b), then $u$ exhibits $Y$ as $W \cdot X$. Thus,
  given $g \colon U \otimes W \rightarrow \C(X,Z)$, we must show that
  $g = \mu \circ (f \otimes u)$ as in~\eqref{eq:38} for a unique
  $f \colon U \rightarrow \C(Y,Z)$. We may take $f$ to be
  \begin{equation}
    U \xrightarrow{1 \otimes \eta} U \otimes W \otimes
    W^\ast \xrightarrow{g \otimes u^\ast} \C(X,Z)
    \otimes \C(Y,X) \xrightarrow{\mu} \C(Y,Z)\rlap{ ;}\label{eq:13}
  \end{equation}
  now that $g = \mu \circ (f \otimes u)$ follows on rewriting with
  the right-hand square of~\eqref{eq:4}, the triangle identities and the
  $\W$-category axioms for $\C$. Moreover, if $f' \colon U \to \C(Y,Z)$
  also satisfies $g = \mu \circ (f' \otimes u)$, then substituting
  into~\eqref{eq:13} gives
  \begin{equation*}
    f = U \xrightarrow{1 \otimes \eta} U \otimes W \otimes
    W^\ast \xrightarrow{f' \otimes u \otimes u^\ast} \C(Y,Z) \otimes \C(X,Y)
    \otimes \C(Y,X) \xrightarrow{\mu \circ (1 \otimes \mu)} \C(Y,Z)
  \end{equation*}
  which is equal to $f'$ via the category axioms for $\C$ and the left
  square of~\eqref{eq:4}.
\end{proof}

Using this result, we may now prove:

\begin{prop}
  \label{prop:27}
  Let $\W$ be a right-closed bicategory. Tensors by $W \in \W(x,y)$
  are absolute if and only if the $1$-cell $W$ admits a left adjoint
  in $\W$.
\end{prop}

\begin{proof}
  If $W$ admits a left adjoint, then the data for a tensor by $W$ can
  be expressed as in Proposition~\ref{prop:17}(b); since these data
  are clearly preserved by any $\W$-functor, tensors by $W$ are
  absolute. Conversely, suppose that tensors by $W$ are absolute; we
  will show that $W$ admits the left dual $[W,I_x] \in \W(y,x)$. The
  counit $\varepsilon$ is the evaluation map
  $\mathrm{ev} \colon [W, I_x] \otimes W \rightarrow I_x$, and it
  remains only to define the unit.

  For each $a \in \W$, we have the $\W$-representation
  $\W(a, \thg) \colon \W \rightarrow \cat{CAT}$ which is closed since
  $\W$ is right-closed. Thus, by the construction of
  Proposition~\ref{prop:29}, there is a tensored $\W$-category
  $\overline{\W(a, \thg)} = a/\W$ whose objects of extent $b$ are
  $1$-cells $a \rightarrow b$, whose hom-objects are
  $(a/\W)(X,Y) = [X,Y]$, and whose tensors are given by
  $Y \cdot X = Y \otimes X$.

  Now, for any $1$-cell $Z \in \W(a,b)$, there is a $\W$-functor
  $[Z, \thg] \colon a/\W \rightarrow b/\W$ given on objects by
  $X \mapsto [Z,X]$ and with action
  $[X,Y] \rightarrow [ [Z,X], [Z,Y] ]$ on hom-objects obtained by
  transposing the composition map in $a / \W$. Since tensors by $W$
  are absolute, they are preserved by
  $[Z, \thg] \colon a/\W \rightarrow b/\W$; it follows that the map
  \begin{equation*}
    \theta_{ZX} \colon W \otimes [Z,X] \rightarrow [Z, W \otimes X]
  \end{equation*}
  in $\W(b,z)$ given by transposing
  $W \otimes \mathrm{ev} \colon W \otimes [Z,X] \otimes Z \rightarrow W
  \otimes X$ is invertible for all $Z \in \W(a,b)$ and
  $X \in \W(a, x)$. In particular, we have
  $\theta_{W,I_x} \colon W \otimes [W,I_x] \cong [W,W \otimes I_x]$ and
  so a unique $\eta \colon I_y \rightarrow W \otimes [W,I_x]$ such that
  $\theta_{VI} \circ \eta$ is the transpose of the morphism
  $\rho_W\lambda_W \colon I_y \otimes W \rightarrow W \rightarrow W
  \otimes I_x$. This condition immediately implies the triangle
  identity $(W \otimes \varepsilon) \circ (\eta \otimes W) = 1$, and
  implies the other triangle identity
  $(\varepsilon \otimes [W,I_x]) \circ ([W,I_x] \otimes \eta) = 1$
  after transposing under adjunction and using bifunctoriality of
  $\otimes$.
\end{proof}

\subsection{Absolute $\lp$-tensors}
\label{sec:absolute-lp-tensors}

Using the above result, we may now characterise the absolute-tensored
$\lp$-categories via the construction $\int$ of
Proposition~\ref{prop:28}.

\begin{prop}
  \label{prop:30}
  A $\lp$-category $\C$ is absolute-tensored if and only if, for all
  $X \in \C_\mathbb{B}$ and all
  $F \colon \mathbb{A} \rightarrow \mathbb{B}$ in $\fl$, there
  exists $Y \in \C_\mathbb{A}$ and a natural isomorphism
  \begin{equation}\label{eq:19}
    \cd[@-0.5em@!C]{
      \mathbb{A} \ar[rr]^-{F} \ar[dr]_-{\iota_Y} &
      \ltwocello{d}{\upsilon}& \mathbb{B} \ar[dl]^-{\iota_X} \\ & \ic
    }
  \end{equation}
\end{prop}

\begin{proof}
  We write $(\thg)_\ast \colon \fl^\mathrm{op} \rightarrow \lp$
  for the identity-on-objects homomorphism sending
  $F \colon \mathbb{A} \rightarrow \mathbb{B}$ to the lex profunctor
  $F_\ast \colon \mathbb{B} \tor \mathbb{A}$ with
  $F_\ast(b,a) = \mathbb{B}(b,Fa)$. Each $F_\ast$ has a left adjoint
  $F^\ast$ in $\lp$ with $F^\ast(a,b) = \mathbb{B}(Fa,b)$, and---as
  all idempotents split in a finitely complete category---the usual
  analysis of adjunctions of profunctors adapts to show that, within
  isomorphism, \emph{every} right adjoint $1$-cell in $\lp$ arises
  thus. So by Proposition~\ref{prop:27}, a $\lp$-category $\C$ is
  absolute-tensored just when it admits all tensors by $1$-cells
  $F_\ast$.

  By Proposition~\ref{prop:17}, this is equally to say that, for all
  $X \in \C_\mathbb{B}$ and $F \in \fl(\mathbb{A}, \mathbb{B})$, we
  can find $Y \in \C_\mathbb{A}$ and maps
  $u \colon F_\ast \rightarrow \C(X, Y)$ and
  $u^\ast \colon F^\ast \rightarrow \C(Y,X)$ rendering commutative
  both squares in~\eqref{eq:4}. To complete the proof, it suffices to
  show that the data of $u$ and $u^\ast$ are equivalent to those of an
  invertible transformation $\upsilon$ as in~\eqref{eq:19}. Now, $u$\pagebreak~
  comprises a natural family of maps
  $\mathbb{B}(j,Fi) \rightarrow \C(X,Y)(j, i)$; equally, by Yoneda,
  elements
  $\upsilon_i \in \C(X, Y)(Fi, i) = \ic(\iota_X(Fi), \iota_Y(i))$
  dinatural in $i \in \mathbb{A}$; or equally, the components of a
  natural transformation $\upsilon$ as in~\eqref{eq:19}. Similar
  arguments show that giving
  $u^\ast \colon F^\ast \rightarrow \C(Y,X)$ is equivalent to giving a
  natural transformation
  $\upsilon^\ast \colon \iota_Y \Rightarrow \iota_X F$, and that
  commutativity in the two squares of~\eqref{eq:4} is equivalent to
  the condition that $\upsilon$ and $\upsilon^\ast$ are mutually
  inverse.
\end{proof}

We now have all the necessary ingredients to prove:

\begin{thm}
  \label{thm:3}
  The $2$-functor
  $\Gamma \colon \pfl \rightarrow \lp\text-\cat{CAT}$
  of~\eqref{eq:34} is an equivalence on hom-categories, and its
  biessential image comprises the absolute-tensored $\lp$-categories.
  Thus $\Gamma$ exhibits $\pfl$ as biequivalent to
  the full and locally full sub-$2$-category of $\lp\text-\cat{CAT}$
  on the absolute-tensored $\lp$-categories.
\end{thm}

\begin{proof}
  By a standard argument, to say that $\Gamma$ is an equivalence on
  homs is equally to say that each counit component
  $\varepsilon_\C \colon \int \Gamma \C \rightarrow \C$
  of~\eqref{eq:34} is an equivalence in $\pfl$. Now, from the
  definitions, the category $\int \Gamma \C$ has:
  \begin{itemize}
  \item \textbf{Objects} being pairs
    $(X \in \S[\mathbb{A}], i \in \mathbb{A})$;
  \item \textbf{Morphisms} $(X,i) \rightarrow (Y,j)$ being maps
    $Xi \rightarrow Yj$ in $\C$;
  \item \textbf{Composition and identities} inherited from $\C$,
  \end{itemize}
  while $\varepsilon_\C \colon \int \Gamma \C \rightarrow \C$ sends
  $(X,i)$ to $Xi$ and is the identity on homsets. So clearly
  $\varepsilon_\C$ is fully faithful; while condition (iii) for a
  left-exact sieve ensures that it is essentially surjective, and so
  an equivalence of categories. However, for $\varepsilon_\C$ to be an
  equivalence in $\pfl$, its pseudoinverse must also be
  partially finite-limit-preserving. This is easily seen to be equivalent to
  $\varepsilon_\C$ being sieve-reflecting; but for each
  $X \in \S[\mathbb{A}]$, the functor
  $\iota_X \colon \mathbb{A} \rightarrow \D$ sending $i$ to $(X,i)$
  and $\varphi$ to $X\varphi$ is by definition in the sieve $\S_{\int\!\Gamma\C}$, and
  clearly $\varepsilon_\C \circ \iota_X = X$.

  This shows that $\Gamma$ is locally an equivalence; as for its
  biessential image, this comprises just those
  $\C \in \lp\text-\cat{CAT}$ at which the unit
  $\eta_\C \colon \C \rightarrow \Gamma(\ic)$ is an equivalence of
  $\lp$-categories. Now, from the definitions, $\Gamma(\ic)$ has:
  \begin{itemize}
  \item \textbf{Objects} of extent $\mathbb{A}$ being functors
    $\mathbb{A} \rightarrow \ic$ in the left-exact sieve
    $\smash{\S_{\int\!\C}}$;
  \item \textbf{Hom-objects} given by
    $\Gamma(\ic)(X,Y) = \ic(X\thg, Y\thg)$;
  \item \textbf{Composition and identities} inherited from $\ic$,
  \end{itemize}
  while the $\lp$-functor $\eta_\C \colon \C \rightarrow \Gamma(\ic)$
  is given on objects by $X \mapsto \iota_X$, and on homs by the
  equality $\C(X,Y) = (\ic)(\iota_X\thg, \iota_Y\thg)$
  of~\eqref{eq:31}; in particular, it is always fully faithful. To
  characterise when it is essentially surjective, note first that
  isomorphisms in $\Gamma(\ic)_\mathbb{A}$ are equally natural
  isomorphisms in $[\mathbb{A}, \ic]$. Now as objects of
  $\Gamma(\ic)_\mathbb{A}$ are elements of
  $\S_{\int\!\C}[\mathbb{A}]$, and since by~\eqref{eq:35} every such
  is isomorphic to $\iota_X F$ for some $X \in \C_\mathbb{B}$ and
  $F \in \fl(\mathbb{A}, \mathbb{B})$, we see that $\eta_\C$ is
  essentially surjective precisely when for all $X \in \C_\mathbb{B}$
  and $F \in \fl(\mathbb{A}, \mathbb{B})$ there exists
  $Y \in \C_\mathbb{A}$ and a natural isomorphism
  $\upsilon \colon \iota_X F \cong \iota_Y$: which by
  Proposition~\ref{prop:30}, happens precisely when $\C$ admits all
  absolute tensors.
\end{proof}

We will use this result in the sequel to freely identify
absolute-tensored $\lp$-categories with partially finitely complete
categories; note that, on doing so, the left $2$-adjoint
$\int \colon \lp\text-\cat{CAT} \rightarrow \pfl$
of~\eqref{eq:34} provides us with a description of the \emph{free
  completion} of an $\lp$-category under absolute tensors.

\section{Lawvere $\A$-theories and their models}
\label{sec:lawv-lp-categ}

We are now ready to give our $\lp$-categorical account of Lawvere
$\A$-theories and their models. We will identify each Lawvere $\A$-theory
with what we term a \emph{Lawvere $\lp$-category on $\A$}, and will
identify the category of models with a suitable category of
$\lp$-enriched functors.

\subsection{Lawvere $\A$-theories as enriched categories}
\label{sec:lawv-lp-categ-1}

Lawvere $\lp$-categories will involve certain absolute-tensored
$\lp$-categories, which in light of Theorem~\ref{thm:3}, we may work
with in the equivalent guise of partially finitely complete
categories. In giving the following definition, and throughout the
rest of this section, we view the finitely complete $\af$ as being
partially finitely complete as in Example~\ref{ex:1}.

\begin{defi}
  \label{def:14}
  Let $\A$ be an lfp category. A \emph{Lawvere $\lp$-category over
    $\A$} comprises a partially finitely complete category $\L$
  together with a map $J \colon \af \rightarrow \L$ in $\pfl$
  which is identity-on-objects and sieve-reflecting. A morphism of
  Lawvere $\lp$-categories is a commuting triangle in $\pfl$.
\end{defi}

\begin{prop}
  \label{prop:26}
  For any locally finitely presentable $\A$, the category of
  Lawvere $\A$-theories is isomorphic to the category of Lawvere
  $\lp$-categories over $\A$.
\end{prop}

\begin{proof}
  Any Lawvere $\A$-theory $J \colon \af \rightarrow \L$ can be viewed
  as a Lawvere $\lp$-category over $\A$ as in Example~\ref{ex:2}; it
  is moreover clear that under this assignation, maps of Lawvere
  $\A$-theories correspond bijectively with maps of Lawvere
  $\lp$-categories. It remains to show that each Lawvere
  $\lp$-category $J \colon \af \rightarrow \L$ over $\A$ arises from a
  Lawvere $\A$-theory. Because in this context, $\af$ is equipped with
  the maximal left-exact sieve, the fact that $J$ is a morphism in
  $\pfl$ is equivalent to the condition that $J \in \S_\L$---so that,
  in particular, $J$ is finite-limit-preserving. Moreover, the fact of
  $J$ being sieve-reflecting implies that $\S_\L$ must be exactly the
  left-exact sieve~\eqref{eq:9} generated by $J$.
\end{proof}

\subsection{Models for Lawvere $\A$-theories as enriched functors}

We now describe how models for a Lawvere $\A$-theory can be understood
in $\lp$-categorical terms. Recall from
Section~\ref{sec:algebr-finit-monads} that we defined $\S$ to be the
$\lp$-category whose objects of extent $\mathbb{A}$ are
finite-limit-preserving functors $\mathbb{A} \rightarrow \cat{Set}$,
and whose hom-objects are given by $\S(X,Y)(i,j) = \cat{Set}(Xi,Yj)$.
By inspection of Definition~\ref{def:17}, this is equally the
$\lp$-category $\Gamma(\cat{Set})$ when $\cat{Set}$ is seen as
partially finitely complete as in Example~\ref{ex:1}.

\begin{prop}
  \label{prop:32}
  For any locally finitely presentable category $\A$, the
  identification of Lawvere $\A$-theories with Lawvere
  $\lp$-categories on $\A$ fits into a triangle, commuting up to
  pseudonatural equivalence:
  \begin{equation*}
    \cd[@!C@C-5em]{
      \cat{Law}(\A)^\mathrm{op}
      \ar[dr]_-{(\thg)\text-\cat{Mod}}\ar[rr]^(0.5){} & \twosim{d} & \sh{r}{1.2em}{(\lp\text-\cat{CAT})^\mathrm{op}}
      \ar[dl]^-{\ \ \ \lp\text-\cat{CAT}(\thg, \S)} \\ & \cat{CAT} &
    }
  \end{equation*}
  wherein the horizontal functor is that sending a Lawvere $\A$-theory
  $J \colon \af \rightarrow \L$ to the $\lp$-category $\Gamma(\L)$.
\end{prop}

\begin{proof}
  We observed in Example~\ref{ex:2} above that, if $J \colon \af
\rightarrow \L$ is a Lawvere $\A$-theory, then the category
$\L\text-\cat{Mod}$ of $\L$-models is isomorphic to the hom-category
$\pfl(\L, \cat{Set})$.
 Since
  $\Gamma \colon \pfl \rightarrow \lp\text-\cat{CAT}$ is an
  equivalence on homs, and since $\Gamma(\cat{Set}) = \S$, we thus
  obtain the components of the required pseudonatural equivalence as
  the composites
  \begin{equation*}
    \L\text-\cat{Mod} \xrightarrow{\cong} \pfl(\L, \cat{Set}) \xrightarrow{\Gamma}
    \lp\text-\cat{CAT}(\Gamma(\L), \S)\text{ .} \tag*{\qEd}
  \end{equation*}
  \def\popQED{}
\end{proof}


\section{Reconstructing the monad--theory correspondence}
\label{sec:monad-theory-corr}

We have now done all the hard work necessary to prove our main result.

\begin{thm}\label{thm:5}
  The process of freely completing a one-object $\lp$-category under
  absolute tensors induces, by way of the identifications of
  Propositions~\ref{prop:8} and~\ref{prop:26}, an equivalence betwen
  the categories of finitary monads on $\A$ and of Lawvere
  $\A$-theories. This equivalence fits into a pseudocommuting
  triangle:
  \begin{equation}\label{eq:3}
    \cd[@!C@C-2em]{
      \cat{Mnd}_f(\A)^\mathrm{op}\ar[dr]_(0.4){(\thg)\text-\cat{Alg}} \ar[rr] &
      \ar@{}[d]|{\textstyle\simeq} &\cat{Law}(\A)^\mathrm{op}\rlap{ .} \ar[dl]^(0.4){\
        \ (\thg)\text-\cat{Mod}}\\ & \cat{CAT}
    }
  \end{equation}
\end{thm}

\begin{proof}
  To obtain the desired equivalence, it suffices by
  Propositions~\ref{prop:8} and~\ref{prop:26} to construct an
  equivalence between the category of $\lp$-categories with a single
  object of extent $\A_f^\mathrm{op}$, and the category of Lawvere
  $\lp$-categories on $\A$.

  On the one hand, given the one-object $\lp$-category $\T$, applying
  the free completion under absolute tensors
  $\int \colon \lp\text-\cat{CAT} \rightarrow \pfl$ to the
  unique $\lp$-functor $! \colon \I_{\af} \rightarrow \T$ yields a
  Lawvere $\lp$-category:
  \begin{equation}\label{eq:1}
    J_\T = \textstyle\int ! \colon \af  \rightarrow \int\!\T\rlap{ ,}
  \end{equation}
  where direct inspection of the definition of $\int$ tells us that
  $\int \I_{\af} = \af$ and that $\int !$ is identity-on-objects and sieve-reflecting.

  On the other hand, if $J \colon \af \rightarrow \L$ is a Lawvere
  $\lp$-category on $\A$, then we may form the composite around the
  top and right of the following square, wherein $\eta$ is a unit
  component of the $2$-adjunction~\eqref{eq:34}:
  \begin{equation}\label{eq:2}
    \cd{
      {\I_{\af}} \ar[r]^-{\eta} \ar[d]_{F} &
      {\Gamma(\af)} \ar[d]^{\Gamma(J)} \\
      {\T_J} \ar[r]^-{G} &
      {\Gamma(\L)\rlap{ .}}
    }
  \end{equation}
  We now factorise this composite as (identity-on-objects, fully
  faithful), as around the left and bottom, to obtain the required
  one-object $\lp$-category $\T_J$.

  The functoriality of the above assignations is direct; it
  remains to check that they are inverse to within
  isomorphism. First, if $\T$ is a one-object $\lp$-category with
  associated Lawvere $\lp$-category~\eqref{eq:1}, then in the
  naturality square
  \begin{equation*}
    \cd{
      {     \I_{\A_f^\mathrm{op}} } \ar[r]^-{\eta} \ar[d]_{!} &
      {\Gamma(\af)} \ar[d]^{\Gamma(\int !)} \\
      {\T} \ar[r]^-{\eta} &
      {\Gamma(\int\! \T)}
    }
  \end{equation*}
  for $\eta$, the left-hand arrow is identity-on-objects, and the
  bottom fully faithful (by Theorem~\ref{thm:3}). Comparing
  with~\eqref{eq:2}, we conclude by the essential uniqueness of
  (identity-on-objects, fully faithful) factorisations that $\T \cong \T_{J_\T}$ as
  required.

  Conversely, if $J \colon \af \rightarrow \L$ is a Lawvere
  $\lp$-category on $\A$ with associated one-object $\lp$-category
  $\T_J$ as in~\eqref{eq:2}, then we may form the following diagram:
  \begin{equation*}
    \cd{
      & {\af} \ar[dl]_-{J_{\T_J} = \int\!F} \ar[dr]^-{J} \\
      {\int\!\T_J} \ar[r]_-{\int\!G} & \int\!\Gamma(\L)
      \ar[r]_-{\varepsilon}&
      {\L}
    }
  \end{equation*}
  where $\varepsilon$ is a counit component of~\eqref{eq:34}. The
  composite around the left and bottom is the adjoint transpose of
  $GF$ under~\eqref{eq:34}; but by~\eqref{eq:2},
  $GF = \Gamma(J) \circ \eta$ which is in turn the adjoint transpose
  of $J$. It thus follows that the above triangle commutes. Since both
  $\int\!F$ and $J$ are identity-on-objects, so is the horizontal
  composite; moreover, $\varepsilon$ is an equivalence by
  Theorem~\ref{thm:3} while $\int\!G$ is fully faithful since $G$ is,
  by inspection of the definition of $\int$. So the lower composite is
  fully faithful and identity-on-objects, whence invertible, so that
  $J \cong J_{\T_J}$ as required.

  We thus have an equivalence as across the top of~\eqref{eq:3}, and
  it remains to show that this renders the triangle below  commutative
  to within pseudonatural equivalence. To this end, consider the
  diagram
  \begin{equation*}
    \cd[@!C@C-4em]{
    \cat{Mnd}_f(\A)^\mathrm{op}\ar[d] \ar[rr] \twocong{drr} & &
    \cat{Law}(\A)^\mathrm{op} \ar[d] \rlap{ .} \\
    \lp\text-\cat{CAT}^\mathrm{op} \ar[rr]^{\Gamma \!\int}
    \ar[dr]_(0.4){\lp\text-\cat{CAT}(\thg,\,\, \S)\ \ }& \ar@{}[d]|{\textstyle\simeq} &
    \lp\text-\cat{CAT}^\mathrm{op} \ar[dl]^(0.4){\ \ \lp\text-\cat{CAT}(\thg,\,\, \S)} \\ &
    \cat{CAT}^\op
  }
  \end{equation*}
  The top square commutes to within isomorphism by our construction of
  the equivalence $\cat{Mnd}_f(\cat{Set}) \simeq \cat{Law}$; whilst
  the lower triangle commutes to within pseudonatural equivalence
  because $\Gamma \!\int$ is a bireflector of $\lp$-categories into
  absolute-tensored $\lp$-categories, and $\S$ is by definition
  absolute-tensored. Finally, by Propositions~\ref{prop:10}
  and~\ref{prop:32}, the composites down the left and the right are
  pseudonaturally equivalent to $(\thg)\text-\cat{Alg}$ and
  $(\thg)\text-\cat{Mod}$ respectively.
\end{proof}

The only thing that remains to check is:
\begin{prop}
  \label{prop:1}
    The equivalence constructed in Theorem~\ref{thm:5} agrees with the
    equivalence constructed by Nishizawa--Power in~\cite{Nishizawa2009Lawvere}.
\end{prop}
\begin{proof}
  We prove this by tracing through the steps by which we constructed
  the equivalence of Theorem~\ref{thm:5}, starting from a finitary
  monad $S \colon \A \rightarrow \A$ in $\cat{LFP}$.

  \begin{itemize}[itemsep=0.25\baselineskip]
  \item We first transport $S$ across the biequivalence
    $\cat{LFP} \simeq \lp$ to get a monad of the form $T \colon \af \tor \af$ in
    $\lp$. From the description of this biequivalence in
    Section~\ref{sec:lex-profunctors}, the underlying lex profunctor
    $T \colon \A_f \times \af \rightarrow \cat{Set}$ is given by
    $T(i,j) = \A(j,Si)$, while the unit and multiplication of $T$ are
    induced by postcomposition with those of $S$.
  \item We next form the one-object $\lp$-category $\T$ which
    corresponds to $T$.
\item We next construct the Lawvere $\lp$-category
  $J \colon \af \rightarrow \int\!\T$ corresponding to $\T$ by
  applying $\int$ to the unique $\lp$-functor
  $\I_{\af} \rightarrow \T$. From the explicit description of $\int$
  in Proposition~\ref{prop:28}, we see that $\int\!\T$ has the same
  objects as $\af$, hom-sets $\int\!\T(i,j) = \A(j,Si)$, and
  composition as in the Kleisli category of $S$. Moreover, the functor
  $J$ is the identity on objects, and given on hom-sets by
  postcomposition with the unit of $S$.
\item Finally, the Lawvere $\A$-theory associated to this Lawvere
  $\lp$-category is obtained by applying the forgetful functor
  $\pfl \rightarrow \cat{CAT}$; which, comparing the preceding
  description with the proof of Theorem~\ref{thm:4}, is exactly the
  Lawvere $\A$-theory associated to the finitary monad
  $S \colon \A \rightarrow \A$. \qedhere
  \end{itemize}
\end{proof}

\end{document}